\documentclass[12pt]{article}
\usepackage{epsfig}
\usepackage{amssymb}
 \newtheorem{theorem}{Theorem}[section]
 
 \newtheorem{corollary}[theorem]{Corollary}
 \newtheorem{remark}[theorem]{Remark}
 \newtheorem{definition}[theorem]{Definition}
\newtheorem{proposition}[theorem]{Proposition}
\newtheorem{example}[theorem]{Example}
\newtheorem{conjecture}[theorem]{Conjecture}

\newtheorem{construction}[theorem]{Construction}
\newtheorem{observation}[theorem]{Observation}
 \newenvironment{proof}{{\it Proof:\/}}{$\Box$\vskip 0.08in}

\newcommand\Z{{\mathbb Z}} 
\newcommand\R{{\mathbb R}}
\newcommand\Q{{\mathbb Q}}
\newcommand\W{{\overline{W}}}

\begin{document}
\pagestyle{myheadings}
 
\thispagestyle{empty}
\markboth {{\sc j.h.przytycki}}
{{\sc Nonorientable, incompressible surfaces }}
\begin{center}
\begin{LARGE}
\baselineskip=10pt
{\bf Nonorientable, incompressible surfaces in punctured-torus
bundles over $S^1$}
\end{LARGE}
\\
\ \\
  J\'OZEF H.~PRZYTYCKI
\end{center}
\centerline{Dedicated to Maite Lozano on the occasion of her 70th birthday}
\begin{quotation}
ABSTRACT.
\baselineskip=10pt
We classify incompressible, $\partial$-incompressible, nonorientable 
surfaces in punctured-torus bundles over $S^1$. We use the ideas of Floyd, Hatcher, and Thurston.
The main tool is to put our surface in the ``Morse position" with respect to the projection
of the bundle into the basis $S^1$.
\end{quotation}
\ \\
\section{Introduction}\label{1}
This paper\footnote{The paper was the second part of the author's doctoral
dissertation prepared at Columbia University under the supervision
of Professor Joan Birman \cite{P-0}. Added for arXiv: After 38 years the paper has been published 
in Revista de la Real Academia de Ciencias Exactas, Físicas y Naturales. Serie A. Matemáticas (RCSM), pp: 1-26; First Online: 26 October 2018.
}
 is devoted to the classification, up to isotopy, of incompressible,
$\partial$-incompressible, nonorientable surfaces in punctured-torus
bundles over $S^1$. We also give a partial classification of nonorientable,
incompressible (not necessary $\partial$-incompressible) surfaces.  In the proof we use
the ideas of A.Hatcher and W.Thurston \cite{H-T} and of
W.Floyd and A.Hatcher \cite{F-H}. The main tool is to put our surface in the ``Morse position" with respect to the projection 
of the bundle into the basis $S^1$. Then we make careful and very laborious analysis of critical points.

We work in the smooth category, however all the results can be proven 
in the PL category as well.
\begin{definition}\label{Definition 1.1.}
\begin{enumerate}
\item
[(a)] Let $M$ be a 3-manifold and $F$ a surface which is either properly
embedded in $M$ or contained in $\partial M$. 
We say that $F$ is 
{\it compressible} in $M$ if one of the following conditions is satisfied:
\begin{enumerate}
\item
[(i)] $F$ is a 2-sphere which bounds a 3-cell in $M$, or
\item
[(ii)] $F$ is a 2-cell and either $F\subset \partial M$ or there is a 3-cell
$X\subset M$ such that $F\subset \partial X$ and  $\partial X\subset (F\cup \partial M)$, or
\item 
[(iii)] there is a 2-cell $D\subset M$ such that $D\cap F =\partial D$ and
$\partial D$ is not contractible in $F$.
\end{enumerate}
We say that $F$ is {\it incompressible} in $M$ if it is not compressible.
\item 
[ (b)] Let $F$ be a submanifold of a manifold $M$. We say that $F$ is
${\pi}_1$-{\em injective} in $M$ if the inclusion induced homomorphism from
${\pi}_1(F)$ to ${\pi}_1(M)$ is an injection.
\item  
[(c)] Let $F$ be a surface properly embedded in a compact 3-manifold $M$,
and ${\partial}_0M$ a component of $\partial M$. We say that $F$ is
$\partial$-{\em incompressible} along ${\partial}_0M$ if there is no 2-disk
$D\subset M$ such that: \ $\partial D\subset ({\partial}_0M\cup F)$,
$D\cap F=\alpha$ is an arc in $\partial D$, $D\cap {\partial}_0M =\beta$
is an arc in $\partial D$. Furthermore $\alpha \cap \beta = \partial\alpha
=\partial\beta$ and $\alpha\cup \beta = \partial D$, and $\alpha$ is not
parallel to $\partial F$ in $F$.  \ We say that $F$ is $\partial$-{\it incompressible} 
in $M$ if $F$ is $\partial$-incompressible along each component of $\partial M$.
\end{enumerate}
\end{definition}

\section{Classification theorems.}\label{2}

In this chapter we prove our main theorem on the structure of incompressible surfaces in a 
punctured-torus bundle over a circle with a hyperbolic monodromy map. We start from the basic construction of Farey diagram 
which can be used to describe the action of the projective special linear group $PSL(2,\Z)$ on the set of fractions $\Q\cup \infty$.
\begin{definition}\label{Definition 2.1.}
\item
[ (a)](\cite{H-T,F-H}) The following graph, $W'$, placed in a disc, 
is called the diagram of $PSL(2,\Z)$ or the Farey diagram.  The set of vertices of $W'$ is 
the set $W=\Q \cup \{\infty\}$, where
$\Q$ is the set of rational numbers. Two vertices $\frac{p_1}{q_1}, 
\frac{p_2}{q_2}\in W$ are joined by an edge if and only if 
$ det \left( \begin{array}{cc}
p_1 & p_2 \\
q_1 & q_2 
\end{array}
\right)
=\pm 1
$
(see Figure 2.1 (a)).
\ \\
\centerline{\psfig{figure=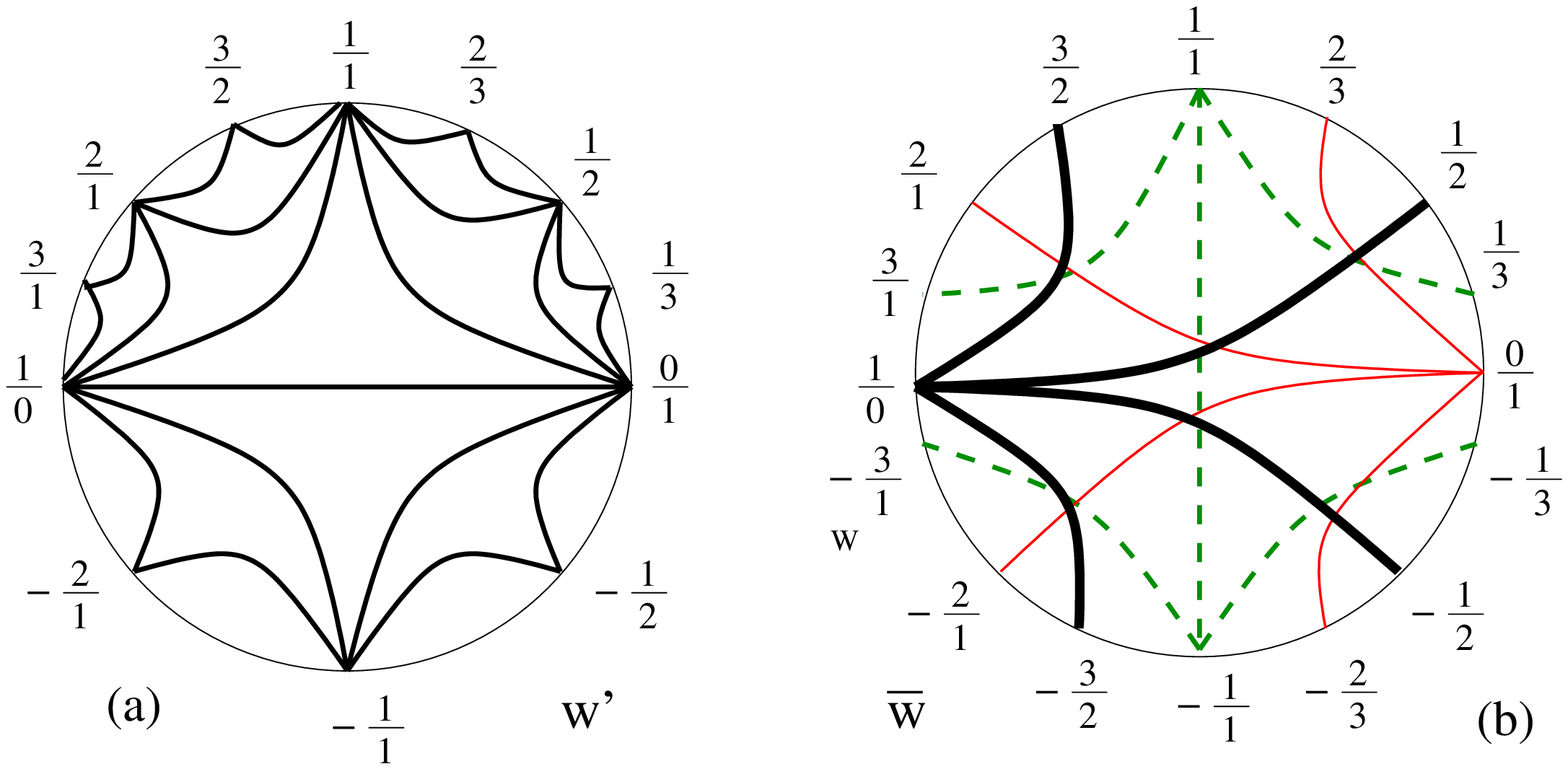,height=7.2cm}}\ \\
\begin{center}
Figure 2.1
\end{center}
 
\item[(b)] Let $W_0=\{\frac{p}{q}\ :\ p \ and \ q \ are\ odd\ \}$, \\
$W_1=\{\frac{p}{q}\ :\ q \ is\ even\}$, \\
$W_2=\{\frac{p}{q}\ :\ p \ is\ even\}$.\\
	We define $\overline {W}$ (respectively $\W_0$, $\W_1$, and $\W_2$)
to be the graph with vertices $W$ (resp. $W_0$, $W_1$ and $W_2$) and such
that two vertices $\frac{p_1}{q_1}$, $\frac{p_2}{q_2}$ in $W$
(resp. $W_0$, $W_1$ and $W_2$) are joined by an edge if and only if
$det \left[ \begin{array}{cc} 
p_1 & p_2 \\ 
q_1 & q_2 
\end{array} 
\right] 
=\pm 2 
$ 
(see Figure 2.1 (b)). 
\end{definition}
We need the following, well-known fact, related, via Cayley graph, to the fact that
matrices $\left[ \begin{array}{cc} 
1 & 2 \\ 
0 & 1 
\end{array} 
\right] 
$ 
and $\left[ \begin{array}{cc}  
1 & 0 \\  
2 & 1  
\end{array}  
\right]  
$  
generate a free subgroup of $PSL(2,\Z)$ (for this classical result see e.g.  \cite{MKS}, Sec. 2.3).

\begin{proposition}\label{Proposition 2.2}
The graph $\W$ is a forest of three connected components, that is: 
$\W=\W_{0} \sqcup \W_{1} \sqcup   \W_{2}$ and $\W_{i}$ ($i=0,1$, or $2$) is a tree.
\end{proposition}
\begin{proof}  The first part follows from the observation: \\*
if $det\left[\begin{array}{cc} 
p_1 & p_2 \\*
q_1 & q_2 
\end{array}\right]
=\pm2$
and $(p_{i},q_{i})=1$ ($i=1,2$), then either $p_{1}$,$p_{2},q_{1},q_{2}$ 
	are odd numbers or $p_{1}$ and $p_{2}$ are even or $q_{1}$ and $q_{2}$ are even.  
	To prove the second part of Proposition \ref{Proposition 2.2} consider the 
edge-path $\frac{1}{0}, \frac{1}{2}, \frac{p_{3}}{q_{3}}, \frac{p_{4}}{q_{4}},...\frac{p_{n}}{q_{n}}$,... in $\W_1$ 
	which is {\it minimal} (i.e. 
$\frac{p_{i+1}}{q_{i+1}} \neq \frac{p_{i-1}}{q_{i-1}}).$  Then $p_{i} \leqslant p_{i+1}$ and $q_{i} < q_{i+1}.$  
This follows by induction on $i$.  Namely $\frac{1}{0}$ and $\frac{1}{2}$ satisfy these inequalities and  we have 
	$\frac{p_{i+1}}{q_{i+1}}=\frac{p_{i-1}+2kp_{i}}{q_{i-1}+2kq_{i}}= \frac{|p_{i-1}+2kp_{i}|}{|q_{i-1}+2kq_{i}|}$ for some integer 
	$k$ depending on $i$, $k\neq 0$.  
	For $k>0$, $p_{i+1}= p_{i-1}+2kp_{i} \geq p_i$, $q_{i+1}= q_{i-1}+2kq_{i} > q_i$  and for negative $k$ (say $-k=k'>0$) we have 
	$p_{i+1}= 2k'p_{i} -p_{i-1} \geq p_i$ and $q_{i+1}= 2k'q_{i} -q_{i-1} > q_i$.  The above inequalities imply that 
no minimal edge path starting at $\frac{1}{0},\frac{1}{2}$ contains a cycle. 
  Generally $\W$ does not contain any cycle by homogeneity 
of $\W.$  Connectivity of $\W_{i}$ (i=0,1,2) may be proved by induction.
\indent
\end{proof}
The following known facts\footnote{In some form one can find it in \cite{B-W,Ru}; it is also written in my PhD thesis \cite{P-0} 
and in \cite{J-P}.}
can be formulated using a notion of a minimal edge-path in $\W.$

\begin{theorem}\label{Theorem 2.3}
Let $F$ be an incompressible surface in $T^{2} \times I$, Then either
\begin{itemize}
\item[(a)] F is isotopic to a saturated annulus (i.e. annulus of type 
$(\gamma) \times I$ for some nontrivial simple closed curve $\gamma$ in $T^{2}$), or
\item[(b)] F is an annulus or torus parallel to the boundary, or
\item[\textbf{(c)}] F is isotopic to a nonorientable manifold uniquely determined by two different slopes 
$\frac{p_{0}}{q_{0}}$ and $\frac{p_{1}}{q_{1}}$ where the determinant of 
$\left[\begin{array}{cc} 
p_0 & p_1 \\ 
q_0 & q_1 
\end{array}\right]$
is even and $F\cap(T^{2} \times \{i\})$ is a curve of slope $\frac{p_{i}}{q_{i}}$ (i=0,1).  The genus 
		of such a surface\footnote{If $F$ is a connected sum of $k$ copies of a projective space 
		and $b$ discs  ($F=$ {\large $\#_k$}$ RP^2 $  {\small\# }  {\large $ \#_b$}$D^2$) 
		then we say that $F$ has genus $k$ and $b$ boundary components. In particular, the Klein bottle has genus $2$.} 
		is equal to the period of the minimal edge-path from 
$\frac{p_{0}}{q_{0}}$ to $\frac{p_{1}}{q_{1}}$ in $\W.$
\end{itemize}
\end{theorem}
\begin{remark}\label{Remark 2.4}
In the case of a surface which is nonorientable and
an isotopy whose restriction to the boundary is the identity 
we have to know additionally the intersection number modulo 2 of $F$ with the
arc {$\{\ast\}$} $\times I$, where {$\{\ast\}$} is a fixed point on $T^{2}$, to determine $F$.
\end{remark}
\begin{corollary}\label{Corollary 2.5}
Each incompressible, non-parallel to the boundary surface in a solid torus $S^{1} \times D^{2}$ is determined, 
up to isotopy by a slope $\frac{p}{q} \in W_{1}.$  The genus of such a surface is equal to the\
period of the minimal edge-path from $\frac{1}{0}$ to $\frac{p}{q}$.  If the period 
is $\textgreater$0 then the surface is nonorientable and $\partial$-compressible.\
\end{corollary}
\begin{corollary}\label{Corollary 2.6}
(\cite{B-W},\cite{Ru}).  Let $L(q,p)$ be a lens space.  Then 
\begin{enumerate}
\item[(i)] if $q$ is odd then $L(q,p)$ does not contain any incompressible surface,
\item[(ii)] if $q$ is even then $L(q,p)$ contains exactly one incompressible surface, which is nonorientable, and of genus 
equal to the period of the minimal edge-path joining $\frac{1}{0}$ to $\frac{p}{q}$ in $\W_1$.
\end{enumerate}
\end{corollary}


\begin{proposition}\label{Proposition 2.7}
Let $M$ be an irreducible 3-manifold, and $F$ a 
closed, $2$-sided, incompressible surface in $int M$.  Let $M'$ be a manifold 
obtained from  $M$ cut open along $F$ ($M'$ may be connected or not), and
$F'=F_{1} \sqcup F_{2}$ ``a trace" of $F$ in $M'$.  Let $S$ be a properly embedded 
surface in $M$, which is transverse to $F$.\\ 
	Further, let $S'$ be a surface obtained from $S$ by cutting open $M$ along $F$ (that is we delete $F$ and compactify $M-F$ by two 
	copies of $F$).  Then: 
\begin{itemize}
\item [(a)]  If $S'$ is incompressible in $M'$ and $\partial$-incompressible along\\
$F'$ then S is incompressible in M.
\item [(a')]  If $S'$ is incompressible, $\partial$-incompressible in $M'$ then S is\\
incompressible, $\partial$-incompressible in M.
\item [(b)]  If $S$ is incompressible in M and $M'$ has two components\
($M_{1}$ and $M_{2}$) then $S$ can be deformed by isotopy in such a way\
(the new embedding is still denoted by $S$) that $S'$ is incompressible and 
$S' \cap M_{1}$ is $\partial$-incompressible along $F' \cap M_{1}$.  If we\
assume additionally that $S$ is $\partial$-incompressible, we can conclude\
also that $S' \cap M_{1}$ is $\partial$-incompressible.
\end{itemize}
\end{proposition}
If $\partial M$ consists of tori, Proposition \ref{Proposition 2.7} and Theorem 2.3 give us:

\begin{proposition}\label{Proposition 2.8}
Let M be a compact, irreducible 3-manifold with $\partial M$ equal to a collection of tori $T_{1},...,T_{k}$.  
	Let $V_{1},...,V_{k}$ be small regular neighborhoods of $T_{1},...,T_{k}$ and 
$M'=M-int  \bigcup_{i=1}^k   V_{i}$ ($M'$ is homeomorphic to $M$ and $V_{i}$ to $T^{2} \times [0,1]$ for each $i$). 
Then each incompressible, non-parallel to the boundary surface 
$S$ ($\partial S \neq \emptyset)$ properly embedded in M can be obtained by gluing together 
an incompressible, $\partial$-incompressible, non-parallel to 
the boundary surface in $M'$, and incompressible, non-parallel to  
the boundary surfaces in $V_{1},...,V_{k}$ (they are described in 
Theorem 2.3).  In particular, if in addition S is orientable  
or has more than one boundary component on each $T_{i}$, then $S$  
is $\partial$-incompressible (see \cite{W}). 
\end{proposition}

The converse to Proposition \ref{Proposition 2.8} is false, i.e. even if $S$ 
allows a decomposition as above it could be compressible (see Example \ref{Example 2.22}).

Now, we will use the above ideas to classify nonorientable, 
incompressible, $\partial$-incompressible surfaces in a punctured-torus  
bundle over $S^{1}$ (compare remarks in the preliminary version of \cite{F-H} circulating around 
1980 when I was a PhD student at Columbia).

First we introduce some models and constructions.\\
Let $F$ be a torus with a hole.  We will use two standard models for $F$:\\
(a) either as a  square with the opposite edges identified and a hole cut in the middle, or\\
(b) as a square $[-1,1]\times [-1,1]$ with corners cut and the level $x=-1$ identified with the level $x=1$, and 
the level $y=-1$ identified with the level $y=1$ (Figure 2.2)\\
\centerline{\psfig{figure=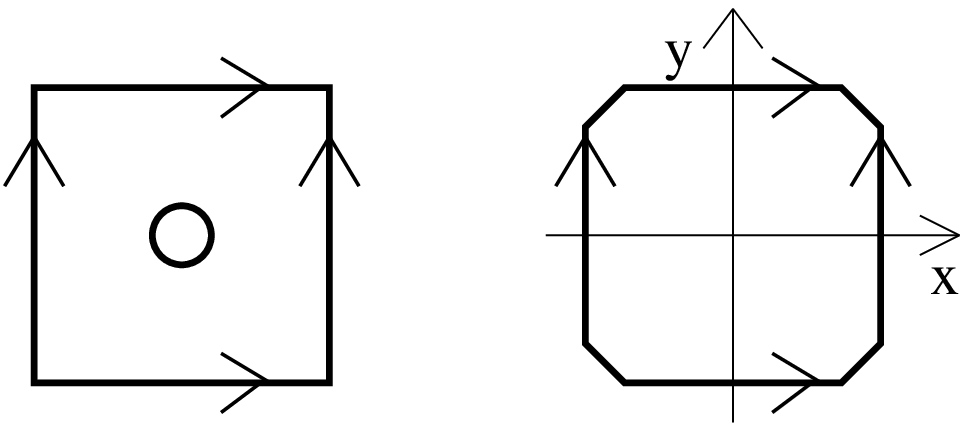,height=2.4cm}}\ \\
\centerline{Figure 2.2; two models for a torus with a hole}

Similarly we will use two ``cube" models for $F \times I$ as a product with interval of models of $F$; see Figure 2.3.
\ \\ 
\centerline{\psfig{figure=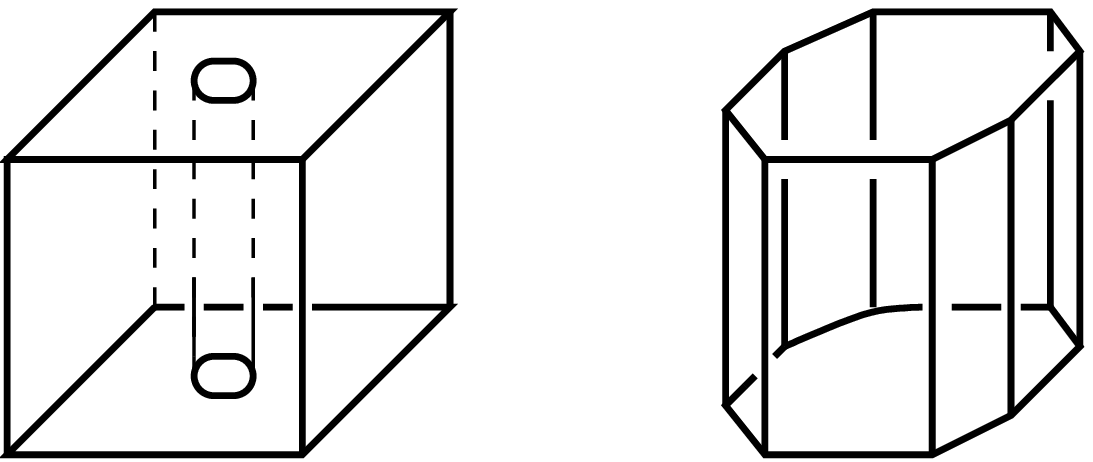,height=3.0cm}} 
\centerline{Figure 2.3; two models for a product $F\times [0,1]$.}
\ \\ \ \\
If a surface $S$ is embedded in $F \times I$ such that the restriction of the natural projection $p:F \times I \rightarrow I$ 
to $S$ (that is $p / S$) is a Morse function, then instead 
of drawing a saddle we will schematically draw the projection of $S$ in the neighborhood of the saddle onto the level of the 
saddle (See Figures 2.4, 2.5 or 2.6).\\
\indent
From now on, we assume that each monodromy map $\phi$ is of the  
form $\phi = \pm\left[\begin{array}{cc} 
a & b \\ 
c & d 
\end{array}\right]$ where $a,b,c,d \geq 0$ (each hyperbolic matrix is  
conjugate to a matrix of such a form).

Recall that elements of $SL(2,Z)$ can be divided into three classes depending on its trace:\\
(e) elliptic if $|tr(\phi)| < 2$, \\
(p) parabolic if $|tr(\phi)| = 2$,\\
(h) hyperbolic if $|tr(\phi)| >2$, or equivalently $\phi$ has two different real eigenvalues.

\begin{construction}\label{Construction 2.9}
Let $\phi:F \rightarrow F$ be a self-homeomorphism of F,\\
defined by $A=\left[\begin{array}{cc} 
a & b \\ 
c & d 
\end{array}\right] \in SL(2,Z)$, i.e. A takes a vector ${y \choose x}$  to
${ay+bx \choose cy+dx}$ \\
          (in particular the slope $\frac{1}{0}$ to $\frac{a}{c}$ and the slope $\frac{0}{1}$ to $\frac{b}{d}$).\\
	  Let $\gamma$ be an edge-path  edge-path in $\W$ (minimal or not) with the successive  vertices 
	 $\ldots , \frac{a_{-1}}{b_{-1}},\frac{a_{0}}{b_{0}},\frac{a_{1}}{b_{1}}, \ldots $. Assume that $\gamma$ is $\phi$-invariant, with period $k$,
	 that is $\phi(\frac{a_{i}}{b_{i}})=\frac{a_{i+k}}{b_{i+k}}$ for all $i$.
Now to each such $\phi$-invariant edge-path $\gamma$,  
we associate the family of surfaces in $M_{\phi}=F \times \R_{/\sim}$ where $(x,t) \sim (\phi(x),t+1)$.  

\begin{itemize}
\item[(a)] First construct surface $\tilde{S_{\gamma}}$ in $F \times \R.$ Let $F_{t}=F \times \{t\}$\\
$\tilde{S_{\gamma}} \cap F_{i/k}$ = the standard circle of slope $\frac{a_{i}}{b_{i}}.$  The saddles are\\
on the levels $\frac{1}{k}(i+\frac{1}{2}),$ for all i.\\

Consider a saddle with the slope $\frac{1}{0}$ below the critical point and the slope $\frac{1}{2}$ above the critical point (Figure 2.4):
\ \\
\centerline{\psfig{figure=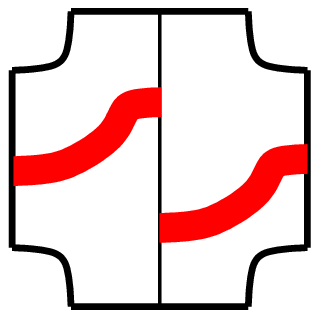,height=3.1cm}}\ \\
\centerline{Figure 2.4; the saddle is changing a circle of slope $\frac{1}{0}$ to a circle of slope $\frac{1}{2}$}
 \ \\
The saddle between $\frac{a_{i}}{b_{i}}$ and $\frac{a_{i+1}}{b_{i+1}}$ in $\tilde{S_{\gamma}}$ is obtained from that 
in Figure 2.4 by applying the homeomorphism given by a linear 
isomorphism which takes the slopes $\frac{1}{0}$ to $\frac{a_{i}}{b_{i}}$ and $\frac{1}{2}$ to $\frac{a_{i+1}}{b_{i+1}}.$ \ 
Now we define a surface $S_{\gamma}^{c}$=$\tilde{S}_{\gamma / \phi},$ in a punctured-torus bundle over $S^1$ with monodromy $\phi$.

\item[(b)] Let $S_{\gamma}^{\partial}$ be a surface obtained from $S_{\gamma}^{c}$ by modifying $S_{\gamma}^{c}$ between 
levels $F_{0}$ and $F_{\frac{1}{2k}}$ by adding one saddle and one horizontal boundary component; see Figure 2.5 (compare Observation 2.18).\\
 \ \\
\centerline{\psfig{figure=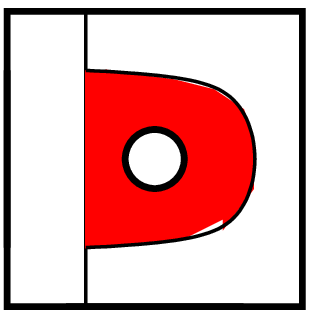,height=3.1cm}}\ \\
\centerline{Figure 2.5; adding a saddle and a horizontal boundary component}
 \ \\
As will be shown later, in Observation \ref{Observation 2.19}, $S_{\gamma}^{\partial}$ is incompressible
but not $\partial$-incompressible. 
\end{itemize}
\end{construction}

\begin{construction}\label{Construction 2.10}
We will define a surface $S_{\gamma}(\varepsilon_{1},...,\varepsilon_{k})$ 
associated to a $\phi$-invariant edge-path $\gamma$ (minimal or not) 
in $\W$ of period $k$, and an element ($\varepsilon_{1},\varepsilon_{2},...,\varepsilon_{k}) \in (\Z_{2})^{k}.$\\
Let $...\frac{a_{-1}}{b_{-1}},\frac{a_{0}}{b_{0}},\frac{a_{1}}{b_{1}}...$ be the successive vertices of $\gamma$ with 
   $\phi(\frac{a_{i}}{b_{i}})=\frac{a_{i+k}}{b_{i+k}}$, and $\varepsilon_{k+i}=\varepsilon_{i}$, for all i.  First we construct
surface $\tilde{S}_{\gamma} (\varepsilon_{1},...,\varepsilon_{k})$ in $F \times \R.$  $\tilde{S} \cap F_{\frac{i}{k}}$ = the standard arc of 
slope $\frac{a_{i}}{b_{i}}.$  The saddles are on the levels $\frac{1}{k}(i+\frac{1}{2}),$ for all i.\\
Consider two schemes of saddles:

 \ \\
\centerline{\psfig{figure=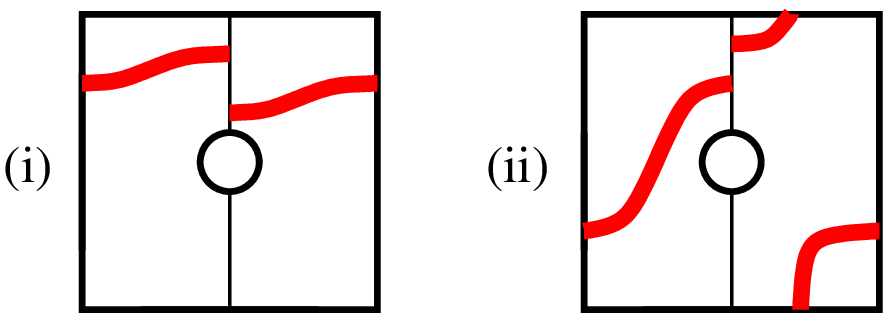,height=3.1cm}}\ \\
\centerline{Figure 2.6; saddles between slopes $\frac{1}{0}$ and $\frac{1}{2}$}

Both saddles lead from arcs of slope $\frac{1}{0}$ to arcs of slope $\frac{1}{2}.$\\
Assume $a_{i}, b_{i}, a_{i+1}, b_{i+1} \geqslant 0.$  The saddle of type $\varepsilon_{i}=0$ (resp. $\varepsilon_{i}=1)$
which leads from an arc of slope $\frac{a_{i}}{b_{i}}$ to an arc of slope $\frac{a_{i+1}}{b_{i+1}}$ 
is obtained from that of Figure 2.6(i) (resp. Figure 2.6(ii)) by 
applying the homeomorphism given by the linear isomorphism 
which takes a vector ${y \choose x}$ to  ${ a_iy +\frac{1}{2}(a_{i+1}-a_i)x \choose b_iy + \frac{1}{2}(b_{i+1}-b_i)x}$ 
 (in particular the slopes 
$\frac{1}{0}$ to $\frac{a_{i}}{b_{i}}$ and $\frac{1}{2}$ to $\frac{a_{i+1}}{b_{i+1}}).$

Now we can finish the construction of $\tilde S_{\gamma}(\varepsilon_{1},...,\varepsilon_{k}).$ Namely, we construct the saddle of 
type $\varepsilon_{i}$ between levels $F_{\frac{i}{k}}$ and 
$F_{\frac{i+1}{k}}$ (first, for $\frac{a_{i}}{b_{i}}, \frac{a_{i+1}}{b_{i+1}} \geqslant 0,$ later using the $\phi$-invariability of 
$\tilde S_{\gamma}(\varepsilon_{1},...,\varepsilon_{k})).$\\
Finally we define $S_{\gamma}(\varepsilon_{1},...,\varepsilon_{k}) = \tilde{S_{\gamma}}(\varepsilon_{1},...,\varepsilon_{k})_{/\phi}.$ 
\end{construction}
Notice that 
the surfaces $\tilde S_{\gamma}(\varepsilon_{1},...,\varepsilon_{k})$ for given $\gamma$ are carried by branched surface 
$\tilde{\Sigma}(\gamma)$ where 
$$\tilde{\Sigma}(\gamma)= \bigcup_{i=-\infty}^{\infty} F_{(i+1/2)/k} \cup \bigcup_{i=-\infty}^{\infty} P_{i},$$        
	where $P_{i}={\alpha}_{i} \times [\frac{i-\frac{1}{2}}{k},\frac{i+\frac{1}{2}}{k}]$ and ${\alpha}_{i}$ 
	is a standard arc of slope $\frac{a_{i}}{b_{i}}.$
	That is, our branched surface contains a punctured-torus on each saddle (critical) level and appropriate slope curves 
	on other levels  (compare \cite{H-T}).

Recall the result of \cite{F-H} that if $\gamma$ is a $\phi$-invariant edge-path in the diagram of $PSL(2,\Z)$ then we can uniquely assign 
to $\gamma$ a $\phi$-invariant surface $\tilde{S_{\gamma}}$ in $F \times \R.$ Let $S_{\gamma}=\tilde{S}_{\gamma/\phi}$ (this definition 
is slightly different than that of \cite{F-H}). 
Let further $\bar{S_{\gamma}}=\partial N(S_{\gamma})$ where $N(S_{\gamma})$ is a tubular neighborhood of $S_{\gamma},$ 
for $\gamma$ of odd period (so $S_{\gamma}$ nonorientable).
\begin{definition}\label{Definition 2.11}
We define, here, a new graph, which we call
the special graph.  The set of vertices of the special 
graph consists of ordered pairs of slopes $(\frac{a}{b},\frac{c}{d})$ which\
satisfy: det $\left[\begin{array}{cc} 
a & c \\ 
b & d 
\end{array}\right]$ = $\pm 1.$  Two vertices $(\frac{a_{1}}{b_{1}},\frac{c_{1}}{d_{1}})$ and $(\frac{a_{2}}{b_{2}},\frac{c_{2}}{d_{2}})$\\
are joined by an edge if and only if either\\
\begin{itemize}
\item[(i)] $\frac{a_{1}}{b_{1}}$ = $\frac{a_{2}}{b_{2}}$ and det $\left[\begin{array}{cc} 
c_{1} & c_{2} \\ 
d_{1} & d_{2} 
\end{array}\right]$ = $\pm 2,$ or\\
\item[(ii)] $\frac{c_{1}}{d_{1}}$ = $\frac{c_{2}}{d_{2}}$ and det $\left[\begin{array}{cc} 
a_{1} & a_{2} \\ 
b_{1} & b_{2} 
\end{array}\right]$ = $\pm2.$\\
\end{itemize}
An edge-path $\gamma$ in the special graph defines two edge-paths\\
in the graph $\W.$  Namely if\\
$\gamma = ...,(\frac{a_{1}}{b_{1}},\frac{c_{1}}{d_{1}}),(\frac{a_{2}}{b_{2}},\frac{c_{2}}{d_{2}}),...$ then\\
$\gamma_{1} = ...,\frac{a_{1}}{b_{1}}, \frac{a_{2}}{b_{2}}, \frac{a_{3}}{b_{3}},...$\ and \
$\gamma_{2} = ...,\frac{c_{1}}{d_{1}}, \frac{c_{2}}{d_{2}}, \frac{c_{3}}{d_{3}},...$\\
(we allow here, for simplicity, repetitions of consecutive slopes).\\
We say that an edge-path $\gamma$ in the special graph is minimal 
if the associated edge-paths $\gamma_{1}$ and $\gamma_{2}$ are minimal in $\W.$
We say that $\gamma$ is $\phi$-invariant if $\phi(\gamma)=\gamma$ or $-\gamma$ ($-\gamma$ is
obtained from $\gamma$ by changing the order of slopes in each \
vertex of $\gamma).$
\end{definition}
\begin{construction}\label{Construction 2.12}
We define a surface, $S_{\gamma}^{sp},$ associated to a\\
$\phi$-invariant (minimal or not) edge path $\gamma$ in the special graph.\\
Let $k$ be a period of $\gamma$ (i.e. $\phi((\frac{a_{i}}{b_{i}},\frac{c_{i}}{d_{i}}))=(\frac{a_{i+k}}{b_{i+k}},\frac{c_{i+k}}{d_{i+k}})$ for all i,\\
or $\phi((\frac{a_{i}}{b_{i}},\frac{c_{i}}{d_{i}}))=(\frac{c_{i+k}}{d_{i+k}},\frac{a_{i+k}}{b_{i+k}})$ for all i).  First construct\\
surface $\tilde{S}_{\gamma}^{sp}$ in $F \times \R.$  Let $F_{t}=F \times \{t\}.$ $\tilde{S}_{\gamma}^{sp} \cap F_{\frac{i}{k}}$ 
	consists of two arcs: one of them of slope $\frac{a_{i}}{b_{i}}$ and the other of slope $\frac{c_{i}}{d_{i}}.$ 
The saddles are on the levels $\frac{1}{k}(i+\frac{1}{2}),$ for all i.\\
	Consider the following saddle (invariant under the matrix $-Id=\left[\begin{array}{cc} 
		-1 & 0 \\ 
		0 & -1
	\end{array}\right]$):\\ \ 

\centerline{\psfig{figure=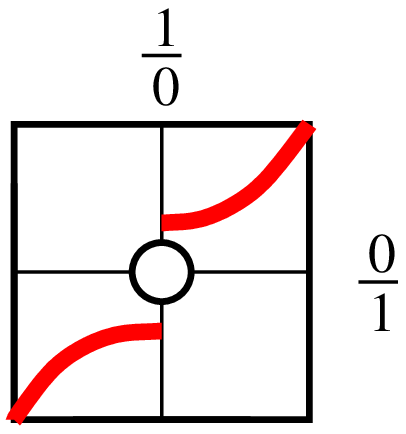,height=3.1cm}}\ \\
\centerline{Figure 2.7; There are two arcs of slopes $\frac{1}{0}$ and $\frac{0}{1}$ under the saddle, }
\centerline{and two arcs of slopes $\frac{1}{2}$ and $\frac{0}{1}$ above the saddle}
\ \\
The saddle (in $\tilde{S}_{\gamma}^{sp}$) associated to an edge $(\frac{a_{i}}{b_{i}},\frac{c_{i}}{d_{i}})(\frac{a_{i+1}}{b_{i+1}},\frac{c_{i+1}}{d_{i+1}})$\\
is obtained from that on Figure 2.7 by applying the homeomorphism defined by the linear isomorphism given by:
\begin{enumerate}
\item[(i)] if the edge is of the type (i) of Definition 2.11:\\
the slope of $\frac{0}{1}$ is taken to $\frac{a_{i}}{b_{i}}$, $\frac{1}{0}$ to $\frac{c_{i}}{d_{i}}$ and $\frac{1}{2}$ 
		to $\frac{c_{i+1}}{d_{i+1}}i=\frac{c_{i}+2a_i}{d_{i}+2b_i};$
\item[(ii)] if the edge is of the type (ii) of Definition 2.11:\\
the slope of $\frac{0}{1}$ is taken to $\frac{c_{i}}{d_{i}}$, $\frac{1}{0}$ to $\frac{a_{i}}{b_{i}}$ and $\frac{1}{2}$ to 
		$\frac{a_{i+1}}{b_{i+1}}= \frac{a_{i}+2c_i}{a_{i}+2d_i}.$
\end{enumerate}
Now we define ${S}_{\gamma}^{sp} = \tilde{S}_{\gamma/\phi}^{sp}.$  If $\phi(\gamma)=\gamma$ then ${S}_{\gamma}^{sp}$ consists of 
two components and if $\phi(\gamma)=-\gamma$ then ${S}_{\gamma}^{sp}$ is connected.\

Let $\gamma$ be a $\phi$-invariant edge-path in the special diagram 
with $\phi(\gamma)=-\gamma.$  $\gamma$ uniquely determines $\phi^{2}$-invariant edge-paths $\gamma'_{1}$ and $\gamma'_{2}$ in the diagram of $PSL(2,\Z)$:\\
if $\gamma = \ldots,(\frac{a_{0}}{b_{0}},\frac{c_{0}}{d_{0}}), (\frac{a_{1}}{b_{1}},\frac{c_{1}}{d_{1}}),\ldots,(\frac{a_{k}}{b_{k}},\frac{c_{k}}{d_{k}}),\ldots$ 
	where\\
$\phi(\frac{a_{i}}{b_{i}},\frac{c_{i}}{d_{i}})=(\frac{c_{i+k}}{d_{i+k}},\frac{a_{i+k}}{b_{i+k}}),$ 
	then $\gamma'_{1} = ...,\frac{a'_{0}}{b'_{0}}, \frac{a'_{1}}{b'_{1}},...,\frac{a'_{2k}}{b'_{2k}},...$\\
where 
\[
	\frac{a'_{i}}{b'_{i}} = 
	\left\{
	\begin{array}{ll}
		\frac{a_{i}}{b_{i}} & \mbox{when $0\leq i \leq k$}, \\
		\frac{c_{i-k}}{d_{i-k}} & \mbox{when $k\leq i \leq 2k$}. 
	\end{array}
	\right.
\]
$\gamma'_{2}$ is defined similarly, using $-\gamma$ in the place of $\gamma$ (so 
$\gamma'_{2} = \phi(\gamma'_{1})).$  In fact $S_{\gamma'_{1}} \cup S_{\gamma'_{2}}$ is the boundary of a regular 
neighborhood of the (non-connected) lifting of $S_{\gamma}^{sp}$ to $M_{\phi^{2}}.$

\end{construction}

Now we are ready to formulate our main theorem.

\begin{theorem}\label{Theorem 2.13}
Let $M_{\phi}$ be a punctured-torus bundle over $S^{1}$ with\\
a hyperbolic monodromy map $\phi$ (as in \cite{F-H}, for convenience, we\\
shall usually not distinguish the open manifold $M_{\phi}$ from its\\
natural compactification obtained by adding a boundary torus).\\
Then:
\begin{itemize}
\item[(a)] Each closed, connected, incompressible surface in $M_{\phi}$ is either 
\begin{itemize}
\item[(i)] a torus parallel to the boundary, or 
\item[(ii)] isotopic to one of 
nonorientable surfaces $S_{\gamma}^{c}$, where $\gamma$ is 
a minimal, $\phi$-invariant edge-path in $\W.$
\end{itemize}
\item[(b)]
Each connected, incompressible surface, S, in $M_{\phi}$ with $\partial S$
parallel to the boundary of a fiber is either an annulus parallel 
to $\partial M_{\phi}$, or
\begin{itemize}
\item[(i)] isotopic to a fiber (then $\partial$-incompressible), or
\item[(ii)] isotopic to one of 
nonorientable surfaces $S_{\gamma}^{\partial},$ where $\gamma$ is a minimal,
$\phi$-invariant edge-path in $\W.$
\end{itemize}
\item[(c)]
Each connected, incompressible, $\partial$-incompressible surface, S,
in $M_{\phi}$ with $\partial S$ ($\neq \emptyset$) transverse to each fiber is either
\begin{itemize}
\item[(i)] isotopic to one of the surfaces $S_{\gamma}$ indexed by a minimal, $\phi$-invariant 
edge-path $\gamma$ in the diagram of $PSL(2,\Z)$, or 
		to $\bar{S}_{\gamma}==\partial N(S_{\gamma})$ (where $N(S_{\gamma})$ is a tubular neighborhood of $S_{\gamma},$ ) 
		where the period of $\gamma$ is odd and $\gamma$ is minimal and\
$\phi$-invariant in the diagram of $PSL(2,\Z)$, or 
\item[(ii)] isotopic to one of the surfaces $S_{\gamma}(\varepsilon_{1},...,\varepsilon_{k})$, where\
$\gamma$ is a minimal $\phi$-invariant edge-path in $\W$, or 
\item[(iii)] isotopic to a surface $S_{\gamma}^{sp}$ associated to a minimal,
$\phi$-invariant edge-path $\gamma$ in the special graph with $\phi(\gamma)=-\gamma.$
\end{itemize}
\end{itemize}
\end{theorem}
\begin{definition}\label{Definition 2.14}
Consider a symbol $\gamma[\varepsilon_{1},...,\varepsilon_{k}]$ where $\gamma$ is a
minimal, $\phi$-invariant edge-path in $\W$ of period $k$ and
$[\varepsilon_{1},...,\varepsilon_{k}]\in(\Z_{2})^{k}.$  This symbol uniquely determines a 
$\phi$-invariant (minimal or not) edge-path $\gamma'$ in the diagram 
of $PSL(2,\Z)$:\\
If $\gamma$ is defined by a sequence $...\frac{a_{-2}}{b_{-2}},\frac{a_{0}}{b_{0}},\frac{a_{2}}{b_{2}},...,\frac{a_{2k}}{b_{2k}},...$ where\\
$\phi(\frac{a_{i}}{b_{i}})=\frac{a_{i+2k}}{b_{i+2k}}$, then $\gamma'$ is defined by the sequence of 
vertices $...\frac{a_{-1}}{b_{-1}},\frac{a_{0}}{b_{0}},...,\frac{a_{2k-1}}{b_{2k-1}},\frac{a_{2k}}{b_{2k}},...$ where\\
$$\frac{a_{2i+1}}{b_{2i+1}}=
\left\{
\begin{array}{ccc}
\frac{\frac{1}{2}(a_{2i+2}-a_{2i})}{ \frac{1}{2}(b_{2i+2}-b_{2i})}  & if & \varepsilon_{i+1}=0 \\
\frac{\frac{1}{2}(a_{2i+2} + a_{2i})}{ \frac{1}{2}(b_{2i+2} + b_{2i})}  & if & \varepsilon_{i+1}= 1.
\end{array}\right. $$

This formula is valid for $\frac{a_{2i}}{b_{2i}}, \frac{a_{2i+2}}{b_{2i+2}} \geq 0.$  We use the\
assumption that $\gamma'$ is $\phi$-invariant, to get all vertices of $\gamma'$\
(compare the remark before Construction 2.9).  In fact $\gamma'$\
is associated with the boundary of a regular neighborhood\\
of $S_{\gamma}(\varepsilon_{1},...,\varepsilon_{k})$ in $M_{\phi}$.\\
In considerations below, we consider a part of $\gamma$ in one period with vertices $\frac{a_{i}}{b_{i}} \geq 0.$ \\
Let $\sigma_i=$ 
	\[
\left\{
\begin{array}{ccc}
1   & if & \frac{a_{2i+2}}{b_{2i+2}} > \frac{a_{2i}}{b_{2i}} \mbox { i.e. the edge } \frac{a_{2i}}{b_{2i}},\frac{a_{2i+2}}{b_{2i+2}} \mbox{ goes to the left in the 
diagram of $PSL(2,\Z$)}\\
-1  & if & \frac{a_{2i+2}}{b_{2i+2}} < \frac{a_{2i}}{b_{2i}} \mbox { i.e. the edge } \frac{a_{2i}}{b_{2i}},\frac{a_{2i+2}}{b_{2i+2}} \mbox{ goes to the 
right in the diagram of $PSL(2,\Z$)} ).
\end{array}\right. 
	\]
To complete the definition we introduce an equivalence relation among symbols\\
$\gamma[\varepsilon_{1},...,\varepsilon_{k}]$ by elementary equivalences:\\
$\gamma[\varepsilon_{1},...,\varepsilon_{i},\varepsilon_{i+1},...,\varepsilon_{k}] \sim \gamma[\varepsilon_{1},...,1-\varepsilon_{i},
1-\varepsilon_{i+1},...,\varepsilon_{k}]$ \\
if $\frac{a_{2i+2}}{b_{2i+2}}=\frac{a_{2i-2}\pm 2a_{2i}}{b_{2i-2}\pm 2b_{2i}}$ and either \\ 
(i) $\varepsilon_{i}=\varepsilon_{i+1}=0$ and $\sigma_{i-1}=-\sigma_{i}$ or\\
 (ii) $\varepsilon_{i}=1-\varepsilon_{i+1}$ and $\sigma_{i-1}=\sigma_{i}.$
\end{definition}
\begin{example}\label{Example 2.15}(Compare \cite{H-T})  
Consider a pair of consecutive essential saddles of a surface $S_{\gamma}(\varepsilon_{1},...,\varepsilon_{k})$ with $\gamma$ 
a minimal, invariant edge-path in $\W.$  Suppose the relative heights 
of these two saddles can be reversed by an isotopy of the 
surface supported between the levels slightly below the first 
saddle and slightly above the second, which does not introduce any other critical points.  
If the two saddles are put on the same level, then there are only three possible configurations, 
up to level preserving isotopy and change of coordinates (Figure 2.8):\

\centerline{\psfig{figure=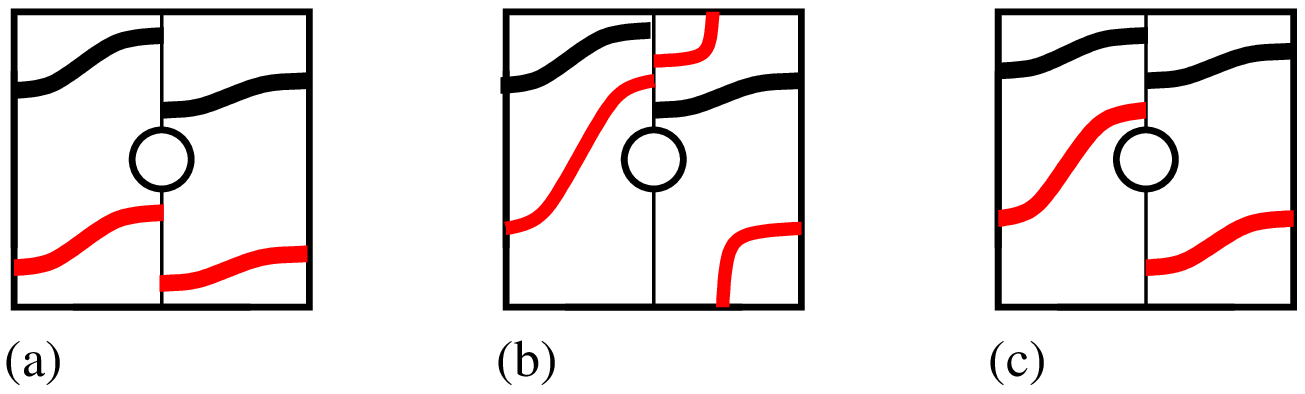,height=3.1cm}}\ \
\centerline{Figure 2.8;  possibility of 2 -saddles on the same level}\ 

\begin{itemize}
\item[(a)] Symbol $\gamma[\varepsilon_{1},...,\varepsilon_{k}]$ remains unchanged.
\item[(b)] The pictures change from\\
\centerline{\psfig{figure=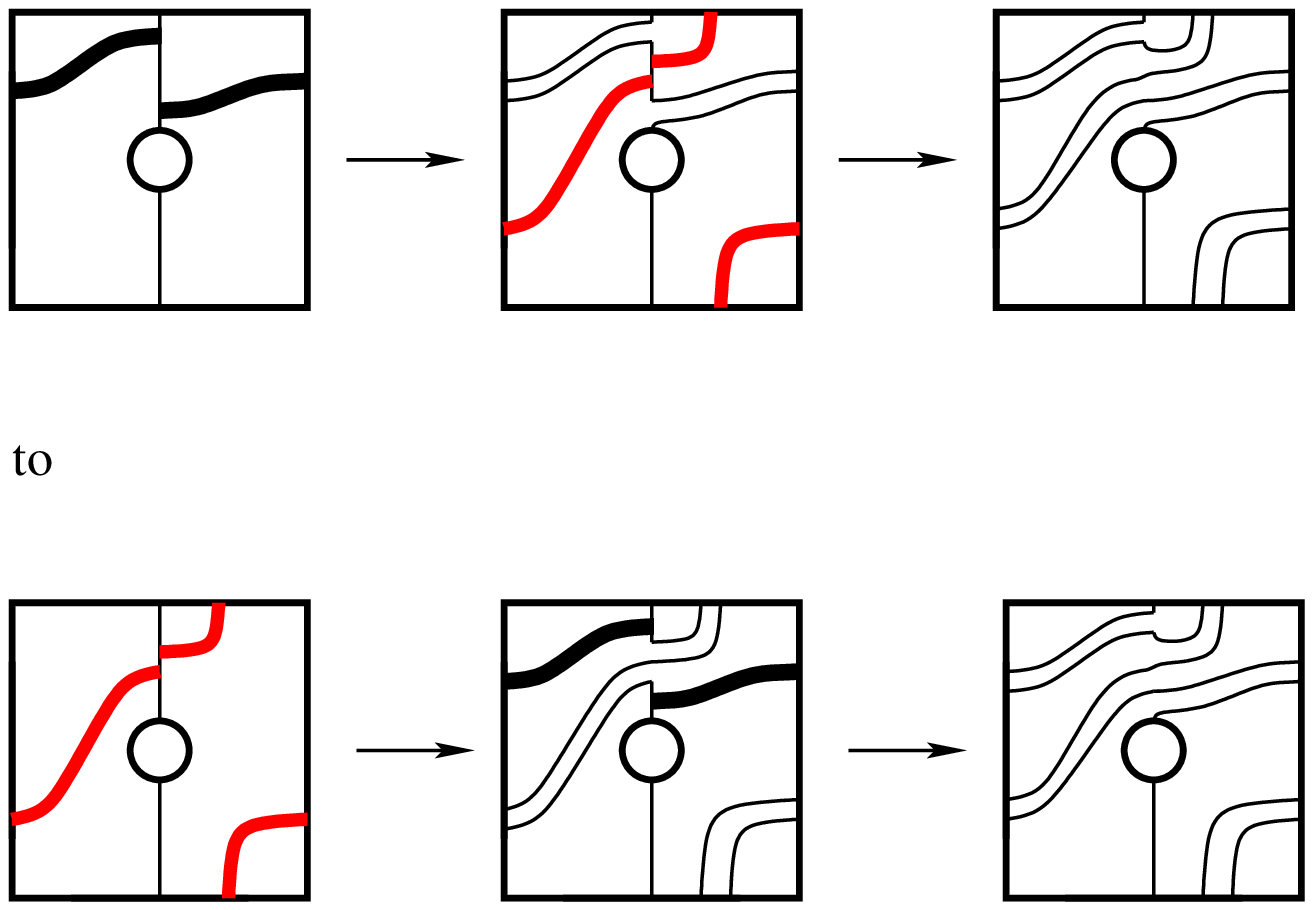,height=7.0cm}}\ \
\centerline{Figure 2.9;  two orders of performing saddles from (b) of Figure 2.8}\

\ \\
$\gamma$ is not changed and the change of $[\varepsilon_{1},...,\varepsilon_{k}]$ reflects the 
equivalence (i) from Definition \ref{Definition 2.14}.\
\item[(c)] The pictures change from \\
 \centerline{\psfig{figure=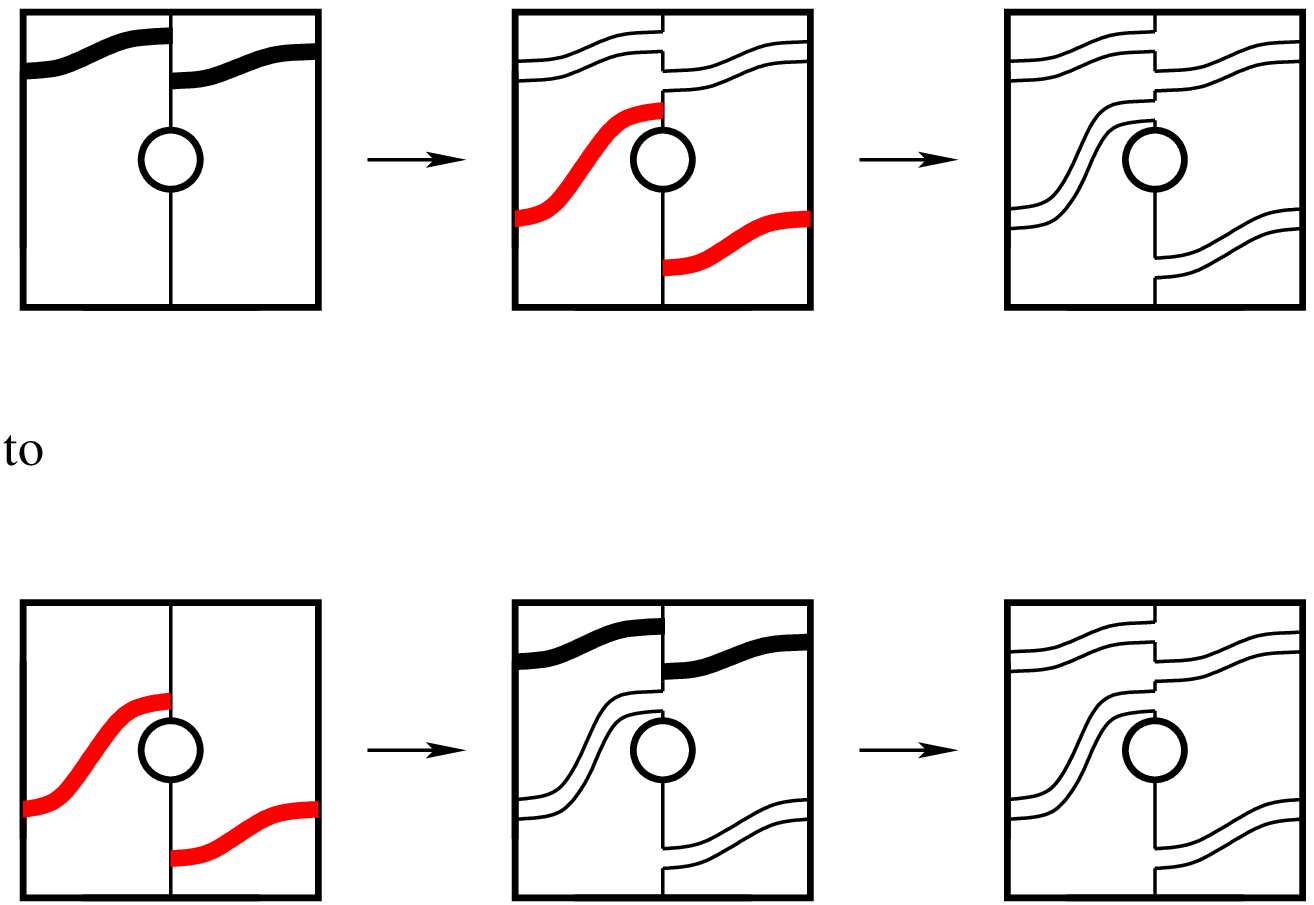,height=7.0cm}}\ 
\centerline{Figure 2.10;  two orders of performing saddles from (c) of Figure 2.8}\

$\gamma$ is not changed and the change of $[\varepsilon_{1},...,\varepsilon_{k}]$ reflects the equivalence (ii) from Definition 2.14.
\end{itemize}
\end{example}

To complete Theorem \ref{Theorem 2.13}, we need:
\begin{proposition}\label{Proposition 2.16}
Let $\gamma$ be a minimal, $\phi$-invariant edge-path in either the diagram of $PSL(2,\Z)$, or $\W$, or in the special graph.\\ Then
surfaces $S_{\gamma},\bar{S_{\gamma}},S_{\gamma}^{c},S_{\gamma}^{\partial},S_{\gamma}(\varepsilon_{1},...,\varepsilon_{k})$ 
and $S_{\gamma}^{sp}$ are incompressible 
(if defined) and furthermore surfaces $S_{\gamma},\bar{S_{\gamma}},S_{\gamma}(\varepsilon_{1},...,\varepsilon_{k})$
and $S_{\gamma}^{sp}$ are $\partial$-incompressible.  Two surfaces from the above 
are isotopic if and only if the following conditions are satisfied:
\begin{itemize}
\item[(i)] the surfaces are associated with the same $\gamma$ (up to
sign, in the case of $S_{\gamma}^{sp}),$
\item[(ii)] they are in the same class 
($S_{\gamma},\bar{S_{\gamma}},S_{\gamma}^{c},S_{\gamma}^{\partial},S_{\gamma}(\varepsilon_{1},...,\varepsilon_{k})$ 
or $S_{\gamma}^{sp}$) and 
\item[(iii)] $[\varepsilon'_{1},...,\varepsilon'_{k}] \sim [\varepsilon_{1}'',...,\varepsilon_{k}'']$ 
if we deal with surfaces of
type $S_{\gamma}(\varepsilon_{1},...,\varepsilon_{k})$ (see Definition \ref{Definition 2.14}).
\end{itemize}
\end{proposition}
Let $g(S)$ denote the genus of a surface $S$, $b(S)$ the number
of boundary curves and $sl(S)$ the slope of $\partial S.$  In this last 
case we have to establish a coordinate system of $H_{1}(\partial M_{\phi})$ to have slope well defined.
The second generator, longitude, of $H_{1}(\partial M_{\phi})$ is determined by 
the boundary of a fiber (with the clockwise orientation; see
Figure 2.11).  To define the first generator, meridian, of $H_{1}(\partial M_{\phi})$
we have to consider two cases (taking into account the fact that for hyperbolic $\phi$ , $|tr(\phi)|>2$):
\begin{itemize}
\item[(a)] $tr \phi \textgreater 0;$ so $\phi$ has two positive eigenvalues.  Then the 
restriction of $\phi$ to the boundary of a fiber ($\partial F$ is
understood to be the set of angles) has four fixed
points, say $\pm \alpha_{1}$ and $\pm \alpha_{2}$.
\\
\centerline{\psfig{figure=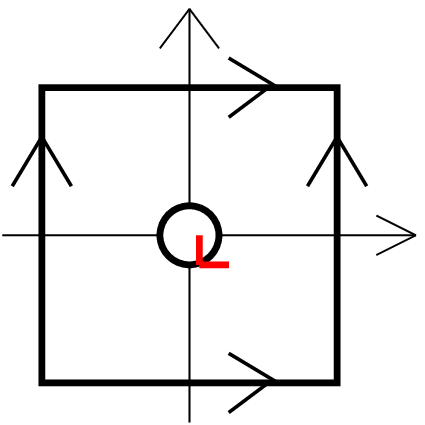,height=4.0cm}}\
\centerline{Figure 2.11;  convention for  orienting the longitude, that is the boundary of the fiber}\

Now the image, under projection $F \times \R \rightarrow M_{\phi}$, of the straight
line in $\partial F \times \R$ which joins $(\alpha_{1},0)$ and $(\alpha_{1},1)$ is a circle which
determines the first generator of $H_{1}(\partial M_{\phi}).$
\item[(b)] $tr \phi \textless 0;$ so $\phi$ has two negative eigenvalues.  Then the 
restriction of $-\phi$ to $\partial F$ has four fixed points, say $\pm \alpha_{1}$
and $\pm \alpha_{2}$, so, in particular, $\phi(\alpha_{1})=-\alpha_{1}.$  Let $\lambda$ be the curve in $\partial F \times \R$ 
given by the equation $z=e^{\pi it}$ where $z \in \partial F$ and $t \in \R$ (so $\lambda$ 
joins $(\alpha_{1},0)$ and $(-\alpha_{1},1)$ with a negative half twist with 
respect to the chosen orientation of $\partial F$).  The image of $\lambda$ 
under projection $F \times \R \rightarrow M_{\phi}$ determines the first generator 
of $H_{1}(\partial M_{\phi}).$  The slope of a curve on $\partial M_{\phi}$ is defined to be
$\frac{\mbox{second coordinate of the curve}}{\mbox{first coordinate of the curve}}.$  
\end{itemize}
\begin{proposition}\label{Proposition 2.17}
The following table establishes\\
dependences among $\gamma$ (of period $k$), $b(S)$, $g(S)$, $sl(S)$. $L$ (or $L_{\gamma}$) denotes the number of ``left turns" in 
$\gamma$; similarly $R$ or ($R_{\gamma}$) denotes the number of ``right turns" in $\gamma$ (see Figure 3.1).
Compare \cite{F-H} Table 1 which we follow partially:

\small{
$$
\begin{array}{lllllllllllll}
\begin{tabular}{c||ccccccccccc} 
$S$ &k & $tr \phi$ & $g(S)$ & $b(S)$ & $sl(S)$ & orientation  \\ 
\hline \hline 
$S_{\gamma}$ & odd  & positive & $k+1$ & 1                      & $\frac{L-R}{4}$   & nonorientable   \\ 
$S_{\gamma}$ & odd  & negative & $k+1$ & 1                      & $\frac{L-R+2}{4}$ & nonorientable   \\
$S_{\gamma}$ & even & positive & $\frac{k}{2}-\frac{b(s)}{2}+1$ & gcd(L-R,4) & $\frac{L-R}{4}$ & orientable   \\
$S_{\gamma}$ & even & negative & $\frac{k}{2}-\frac{b(s)}{2}+1$ & gcd (L-R+2,4)&  $\frac{L-R+2}{4}$ & orientable  \\
\hline
$\bar S_{\gamma}$ & odd & positive  & k  & 2 & $\frac{L-R}{4}$   & orientable  \\ 
$\bar S_{\gamma}$ & odd & negative  & k  & 2 & $\frac{L-R+2}{4}$ & orientable  \\
\hline
$S_{\gamma}^c$          & any & any & $2+k$ & 0  & none & nonorientable\\ 
\hline
$S_{\gamma}^{\partial}$ & any & any & $2+k$ & 1  &  $\frac{1}{0}$ & nonorientable \\
 \hline
$S_{\gamma}(\varepsilon_1,...,\varepsilon_k)$ & any  & positive & $k+2-b(S)$ & $gcd((\Sigma\sigma_i\varepsilon_i),2)$ 
& $\frac{\Sigma_{i=1}^k\sigma_i\varepsilon_i}{2} $ & nonorientable  \\
$S_{\gamma}(\varepsilon_1,...,\varepsilon_k)$ & any  & negative &   $k+2-b(S)$ & 
$gcd((\Sigma\sigma_i\varepsilon_i)+1,2)$ & $\frac{(\Sigma_{i=1}^k\sigma_i\varepsilon_i)+1}{2} $ & nonorientable \\
\hline
$S_{\gamma}^{sp}$       & any &  positive & $k+2-b(S)$ 
& $\frac{gcd(L_{\gamma_1'}-R_{\gamma_1'},4)}{2}$ & $\frac{L_{\gamma_1'}-R_{\gamma_1'}}{8}$ & nonorientable\\ 
with & & \\
$\phi(\gamma)=-\gamma$ & any & negative & $k+2-b(S)$  
& $\frac{gcd(L_{\gamma_1'}-R_{\gamma_1'},4)}{2}$ & $\frac{L_{\gamma_1'}-R_{\gamma_1'}+4}{8}$ & nonorientable\\
\hline
\end{tabular}  
\end{array}$$
}
\centerline{Table 2.1}

\end{proposition}
\ \\

We are now ready to continue our proof of  Theorem \ref{Theorem 2.13}.

\begin{proof}
(a) Let $S$ be a closed, incompressible surface in $M_{\phi}.$  We may 
assume that the projection of S to $S^{1},$ the base of the bundle 
$M_{\phi} \rightarrow S^{1},$ is a Morse function (with all critical points in
	distinct fibers).  Let $F_{t}$ ($t\in S^{1}$) be a fiber transverse to S.
The circles of $S \cap F_{t}$ can be of three types:
\begin{itemize}
\item[(i)] trivial - bounding a disc in $F_{t}$, or 
\item[(ii)] peripheral - isotopic to $\partial F_{t}$, or 
\item[(iii)] essential - representing a non-zero class in $H_{1}(F_{t}).$
\end{itemize}

If an essential circle represents $\pm(p,q) \in H_{1}(F_{t})$ (with
respect to some chosen basis for $H_{1}(F_{t})),$ 
we call $\frac{q}{p} \in \Q \cup \frac{1}{0}$
the slope of the essential circle.  Since $\phi$ is hyperbolic,
no slope can be invariant under $\phi,$ and the only possibility 
of changing slope (without passing by a level without an essential circle in $F_t\cap S$) 
is using a saddle of type:\\

\centerline{\psfig{figure=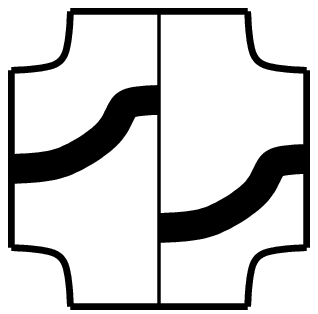,height=3.0cm}}\
\centerline{Figure 2.12;  a saddle between $\frac{q_i}{p_i}=\frac{1}{0}$ and $\frac{q_{i+1}}{p_{i+1}}=\frac{1}{2}$ }\

This saddle could occur if above and below the saddle (near the level of the saddle) there is exactly one curve of 
nontrivial slope (Figure 2.12).\\
Furthermore the slopes satisfy det $\left[\begin{array}{cc} 
q_{i} & q_{i+1} \\ 
p_{i} & p_{i+1} 
\end{array}\right]$ = $\pm2,$ where $\frac{q_{i}}{p_{i}}$ 
is the slope of S below the saddle and $\frac{q_{i+1}}{p_{i+1}}$ above it.\\
So we have two possibilities:
\begin{itemize}
	\item[1.] The slope is undefined in some level $F_{t_{0}}$; that is $F_{t_0} \cap S$ has no esential circle.\\
Now, the proof as in \cite{F-H} works even in the case of nonorientable surfaces.  
The case 1. describes tori which are
parallel to the boundary. 
\item[2.] The slope is defined on each level.\\
We call a saddle essential if the slope does in fact change 
(the only possible type of essential saddle is sketched in Figure 2.12).  
		In fibers $F_{t}$ near an essential saddle trivial 
circles can be eliminated by isotopy of $S$.
Let $F_{t_{1}},...,F_{t_{k}}$ be levels just below essential saddles, and 
$F_{t'_{1}},...,F_{t'_{k}},$ just above essential saddles.  We assume that 
there are no more saddles between $F_{t_{i}}$ and $F_{t'_{i}}.$  Let $n=\Sigma n_{i}$ 
where $n_{i}$ is the number of peripheral circles in $F_{t_{i}} \cap S.$
We prove Theorem \ref{Theorem 2.13} (a) by induction on $n$.
\begin{itemize}
\item[I.] Let $n=0.$  We present the region between $F_{t'_{i}}$ and $F_{t_{i+1}}$\\
as a cube with opposite lateral faces identified and the open\\
neighborhood of the central vertical axis deleted (Figure 2.13).
		We can assume that circles of $F_{t_i'}\cap S$ and $F_{t_{i+1}}\cap S$ have slope 
		$\frac{1}{0}$ and, because $n=0$, both are disjoint from the rectangle $R$ of  Figure 2.13 (a).
\ \\ \ \\
\centerline{\psfig{figure=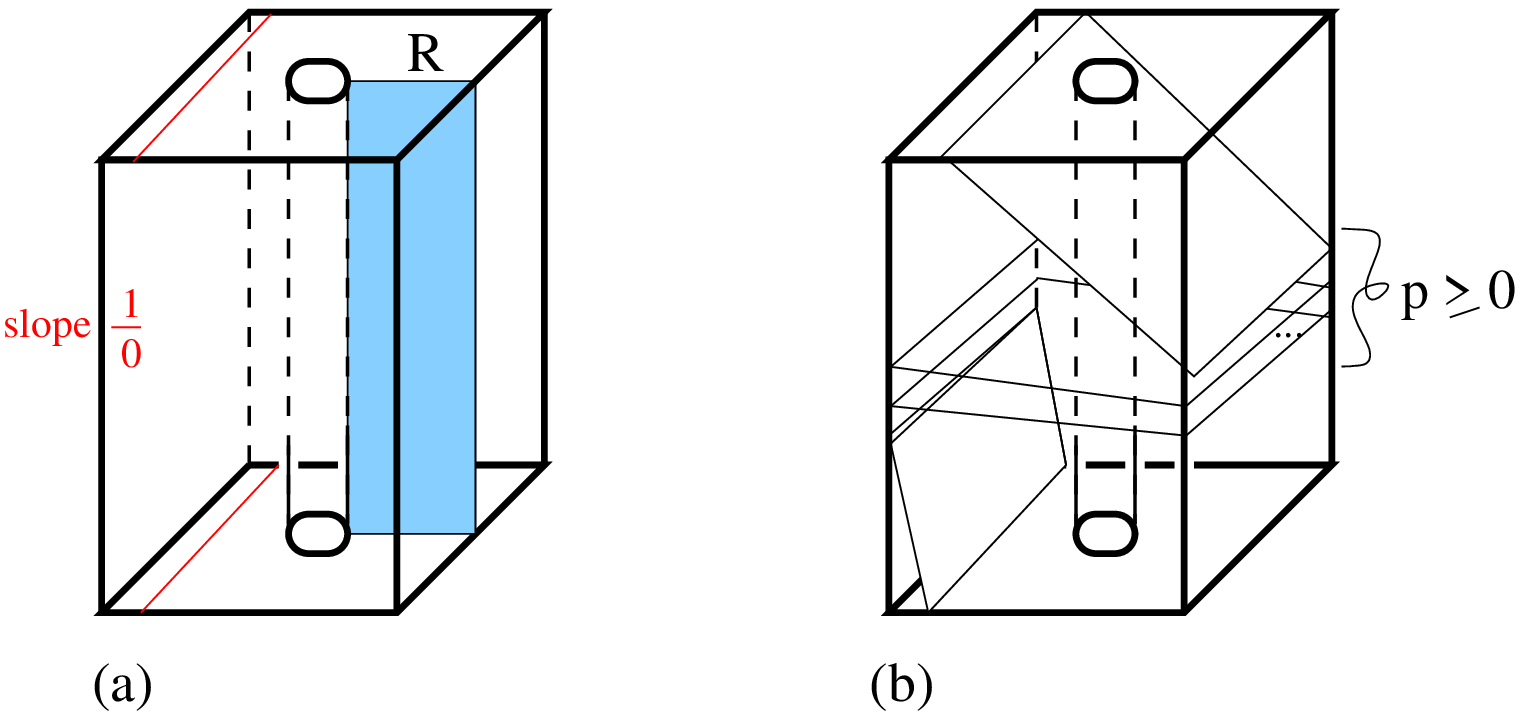,height=5.0cm}}\ \\ \
\centerline{Figure 2.13;  region between $F_{t_{i}'}$ and $F_{t_{i+1}}$}\

Trivial circles of the intersection of $S$ and a lateral face can be eliminated by 
isotopy of $S$.  An arc with endpoints on the same vertical 
edge can be eliminated too.  So, we have the situation as in Figure 2.13 (b).

Now we consider $R\cap S$.  Again, the circles of $R \cap S$ can be eliminated; and the 
arcs with the endpoints on the right edge of $R$ can be eliminated too (Figure 2.14).\\
Thus we can assume that $R \cap S=\emptyset.$\\ 
 \ \\
\centerline{\psfig{figure=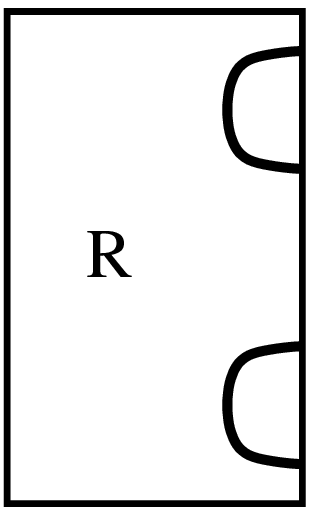,height=3.6cm}}\ \\ \
\centerline{Figure 2.14;  intersection of a surface with the rectangle $R$}\

Therefore we can conclude that no saddle occurs between $F_{t'_{i}}$\\
and $F_{t_{i+1}}.$  This ends the proof of the case when $n=0.$ 

\item[II.] Now let us assume that for each number less than $n>0$
Theorem 2.13 (a) is proven.\\
Consider an embedded surface $S$ with the number of peripheral 
circles equal to $n$.  We will isotope $S$ to decrease $n$.
Analyzing the situation, as before, we have:\\

\centerline{\psfig{figure=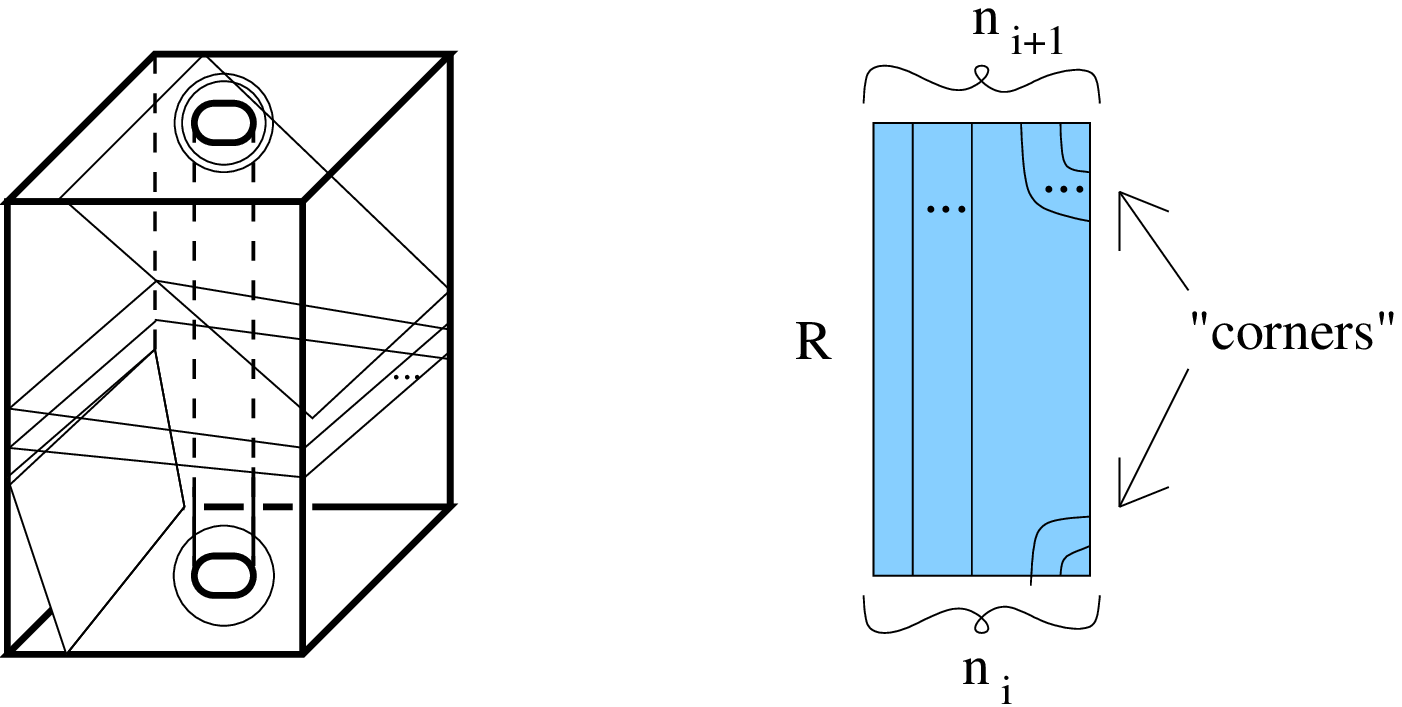,height=5.1cm}}\ \
\centerline{Figure 2.15;  $S$ with $n$ peripheral circles; $R$ with corners}

If $S$ is connected and $n \textgreater 0$, then for some $i$ the corresponding $R$ must contain a ``corner".
		The corner arc gives the following saddle, Figure 2.16(b), which
we will push in the direction of the corner.\

\centerline{\psfig{figure=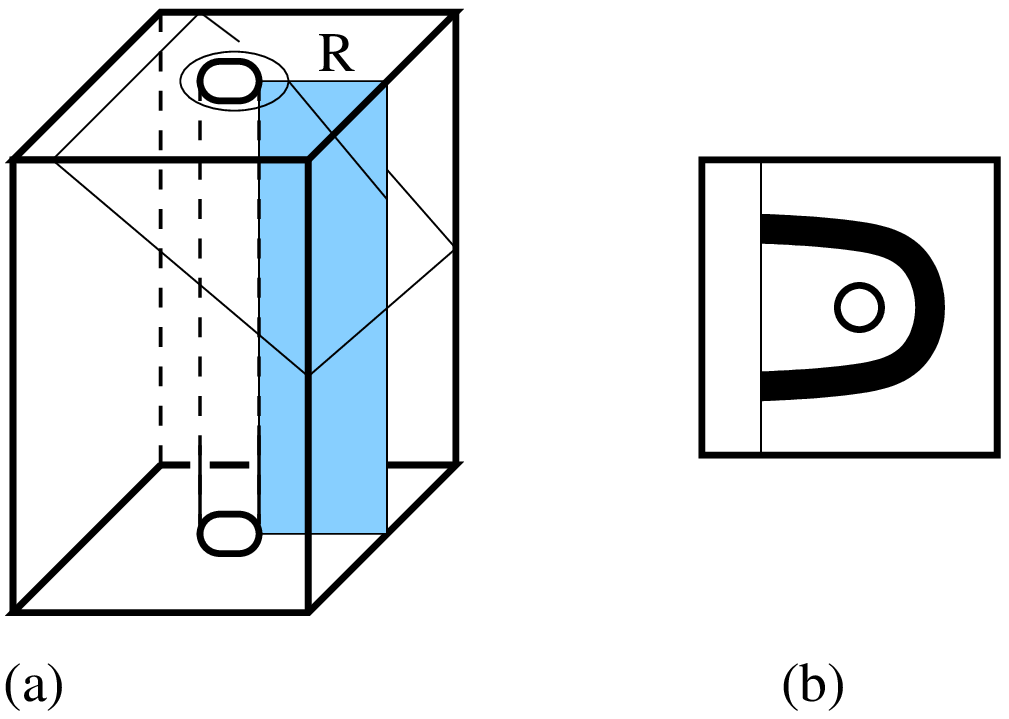,height=5.1cm}}\ 
		\centerline{Figure 2.16;  a corner in $R$ and the corresponding unessential saddle}\

Now we apply the fact that such a saddle commutes with an essential saddle (Figure 2.17) to decrease $n$. 
We decrease $n$ till all unessential saddles are in one box but then $n=0$,
otherwise $S$ is compressible or not connected. If $\gamma$ is not minimal then $S_{\gamma}^{C}$ is easily seen to be compressible. 
\end{itemize}
\end{itemize}
It ends the proof of Theorem 2.13 (a).
\end{proof}
\centerline{\psfig{figure=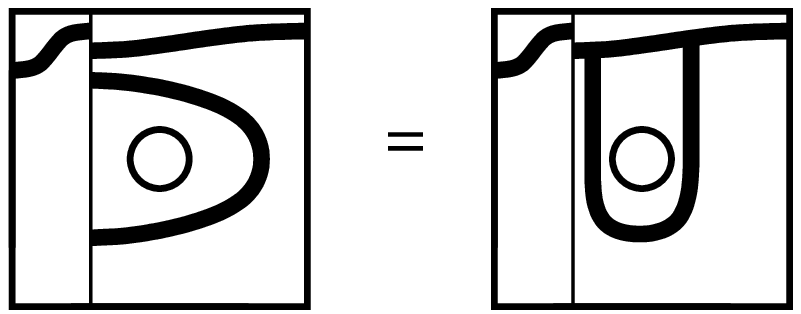,height=3.5cm}}\ \\ 
\centerline{Figure 2.17;  commuting saddles}\

Proof of Theorem 2.13 (b). \\
Let $S$ be a connected, incompressible surface in $M_{\phi}$ with $\partial S$
parallel to the boundary of a fiber.
 Considerations similar to those in 
case (a) lead us to conclusion that we can rearrange the 
saddles in such a way that all unessential saddles and boundary curves 
lie between some essential saddles, say $s_{1}$ and $s_{2}$ 
(in the case of only one essential saddle $s_{1}=s_{2}).$ 
Now if a pair of consecutive saddles looks as in Figure 2.18 (each 
of them ``adds" one horizontal boundary component):

\centerline{\psfig{figure=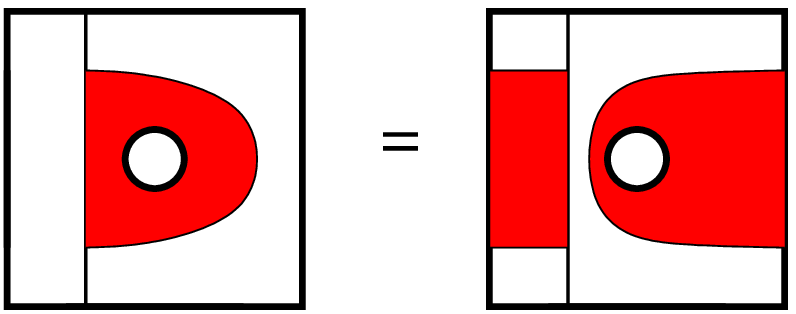,height=3.5cm}}\ \\ \
\centerline{Figure 2.18;  }\
\ \\ 
then the surface is $\partial$-compressible and therefore compressible (Proposition 2.8).  
Thus the region between essential saddles $s_{1}$ and $s_{2}$ looks as follows:\\

\centerline{\psfig{figure=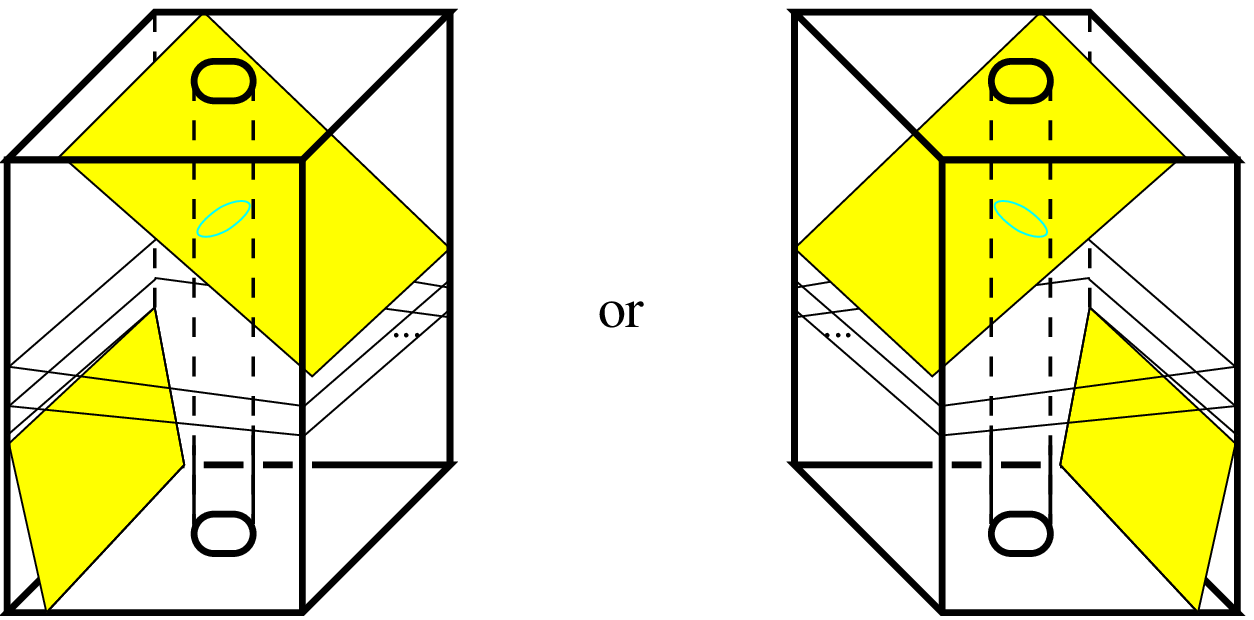,height=3.5cm}}\ \\ \
\centerline{Figure 2.19; region between essential saddles  }\

The following observations end the proof of Theorem 2.13 (b) 
(It will remain only to show that $S_{\gamma}^{\partial}$ is compressible when $\gamma$
is not minimal; but it is easy). 

\begin{observation}\label{Observation 2.18}\ 
(Changing of a direction of an unessential 
saddle).\\
Consider a part of consecutive saddles (in $\tilde{S} \subset F \times \R$):\\
First of them (on the level $i+\frac{1}{2}$) is unessential, and ``adds" a 
horizontal boundary component (Figure 2.20 (a)) and second 
(on the level $i+\frac{3}{2}$) is essential and changes the slope 
	from $\frac{1}{0}$ to $\frac{1}{2}$ (Figure 2.20 (b)).\\


\centerline{\psfig{figure=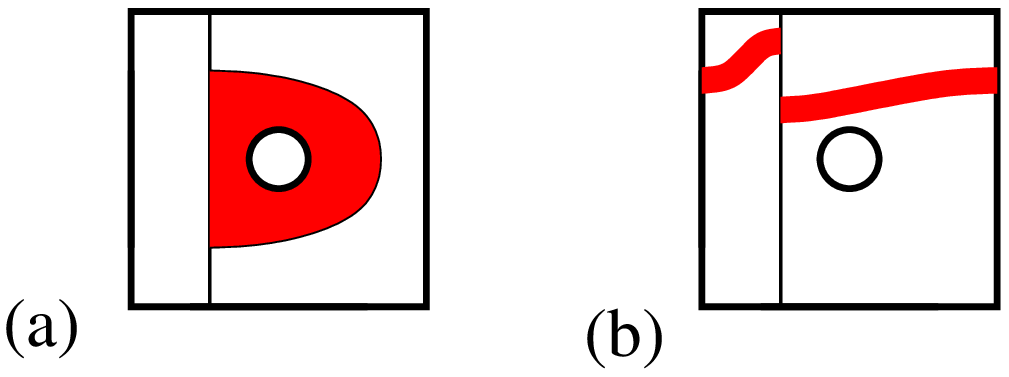,height=3.3cm}}\ \\ \
\centerline{Figure 2.20;  }\
\ \\
Then $\tilde{S}$ can be isotoped (the isotopy is the identity map outside 
the segment $F \times [i,i+2]$) in such a way that the new position 
of $\tilde{S}$ in $F \times [i,i+2]$ is defined by the saddles:\\

\centerline{\psfig{figure=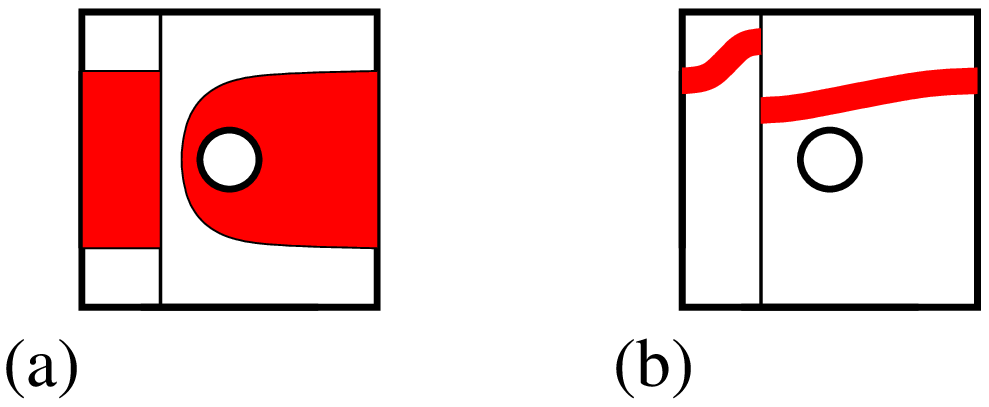,height=3.4cm}}\ \\ \
\centerline{Figure 2.21;  (a) on the level $i+\frac{1}{2}$ and (b) on the level $i+\frac{3}{2}$}\
\ \\
The proof follows from the fact that one can change the order 
of the unessential saddle (Figure 2.20 (a) or 2.21 (a)) and the  
essential saddle (Figure 2.20 (b)); as shown in Figure 2.22.
\end{observation}

\centerline{\psfig{figure=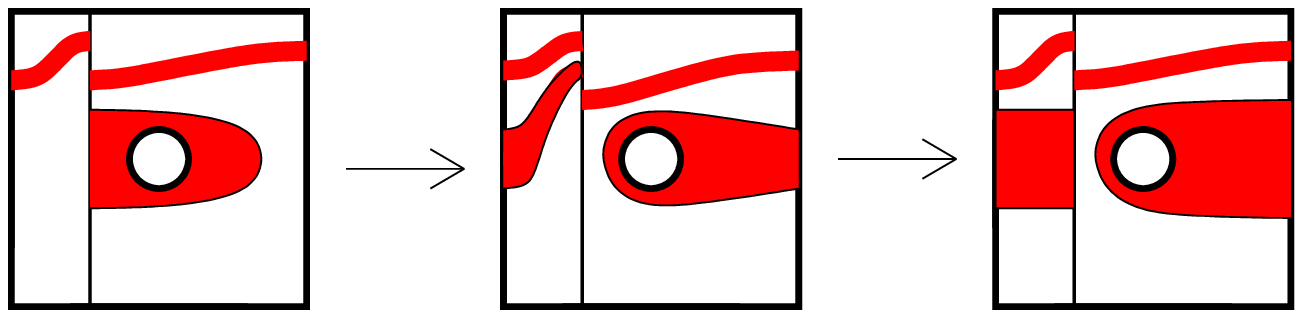,height=3.3cm}}\ 
\centerline{Figure 2.22;  }\
\ \\
\begin{observation}\label{Observation 2.19}
The surface $S_{\gamma}^{\partial}$ is $\partial$-compressible when $\gamma$ is a\\
minimal edge-path of positive period in $\W$.  $S_{\gamma}^{\partial}$ can be obtained from each surface of type 
$S_{\gamma}(\varepsilon_{1},...,\varepsilon_{k})$ with 
$b(S_{\gamma}(\varepsilon_{1},...,\varepsilon_{k}))=1$ by construction described in Proposition \ref{Proposition 2.8}. 
Namely: consider $S_{\gamma}(\varepsilon_{1},...,\varepsilon_{k}) \subset M_{\phi}$ with  $b(S_{\gamma}(\varepsilon_{1},...,\varepsilon_{k}))=1$.\\
$sl(S_{\gamma}(\varepsilon_{1},...,\varepsilon_{k}))$ is of type $\frac{2m+1}{2}$.\\
Consider $S_{0}$, a Mobius band with a hole, embedded properly in\\ 
$T^{2} \times [0,1]$ such that $S_{0} \cap T^{2} \times \{0\}=a$ curve of slope $\frac{2m+1}{2}$ and\\
$S_{0} \cap T^{2} \times \{1\}=a$ curve of slope $\frac{1}{0}$.  $S_{0}$ is incompressible (compare Theorem \ref{Theorem 2.3}).  
Now we glue $(M_{\phi},S_{\gamma}(\varepsilon_{1},...,\varepsilon_{k}))$ with
$(T^{2} \times [0,1], S_{0})$ along $(\partial M_{\phi}, \partial S_{\gamma}(\varepsilon_{1},...,\varepsilon_{k}))$ and 
$(T^{2} \times \{0\}, T^{2} \times \{0\} \cap S_{0}$).  The new manifold with $S_{\gamma}(\varepsilon_{1},...,\varepsilon_{k}) \cup S_{0}$
embedded is homeomorphic to $(M_{\phi},S_{\gamma}^{\partial})$ (see Figure 2.23).
\end{observation}

\centerline{\psfig{figure=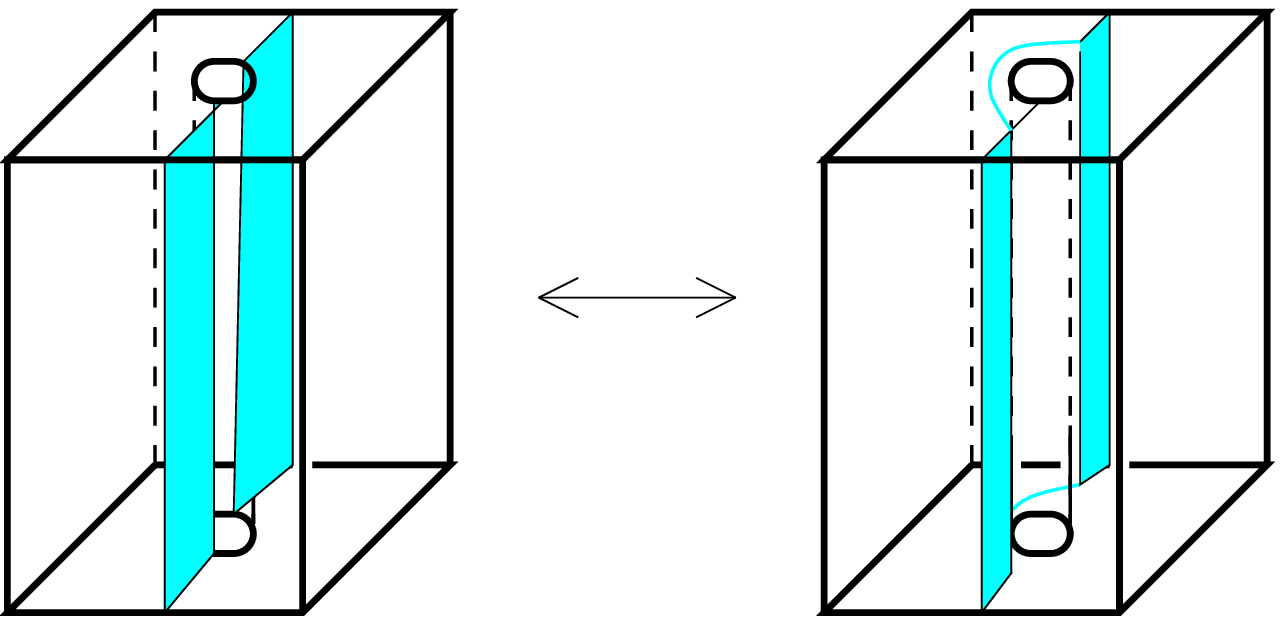,height=5.5cm}}\ \\ \
\centerline{Figure 2.23;  }\

Proof of Theorem \ref{Theorem 2.13} (c).  Let $S \subset M_{\phi}$ be a compact, incompressible, 
$\partial$-incompressible surface with $\partial S \neq \emptyset$, the circles of $\partial$S 
not being isotopic to fibers in $\partial M_{\phi}$.  Then $S$ can be isotoped 
so that $\partial$S is transverse to the fibers in $\partial M_{\phi}$ and the bundle 
projection is a Morse function on $S$.  In a non-critical fiber 
$F_{t}$, the arcs of $S \cap F_{t}$ must all be non-trivial in $H_{1}(F_{t},\partial F_{t})$, 
since a homologically trivial arc would bound a disk on $F_{t}$ 
(so $F$ would be $\partial$-compressible).
\ \\ \ \\
\centerline{\psfig{figure=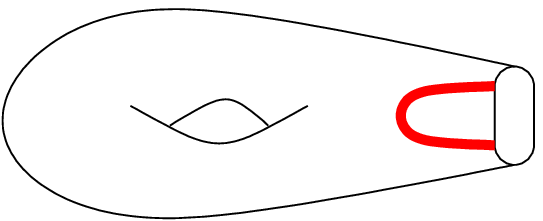,height=3.0cm}}\ \\ \
\centerline{Figure 2.24;  an arc parallel to the boundary of a fiber} \

The same number of arcs with a defined slope is on each \
non-critical level.\\
As before (\cite{F-H}) it can not happen that $S \cap F_{t}$ has three-slope configuration.  So the only possible changes of slope are 
of type (for better visualization we draw two models of a punctured-torus):\\

\centerline{\psfig{figure=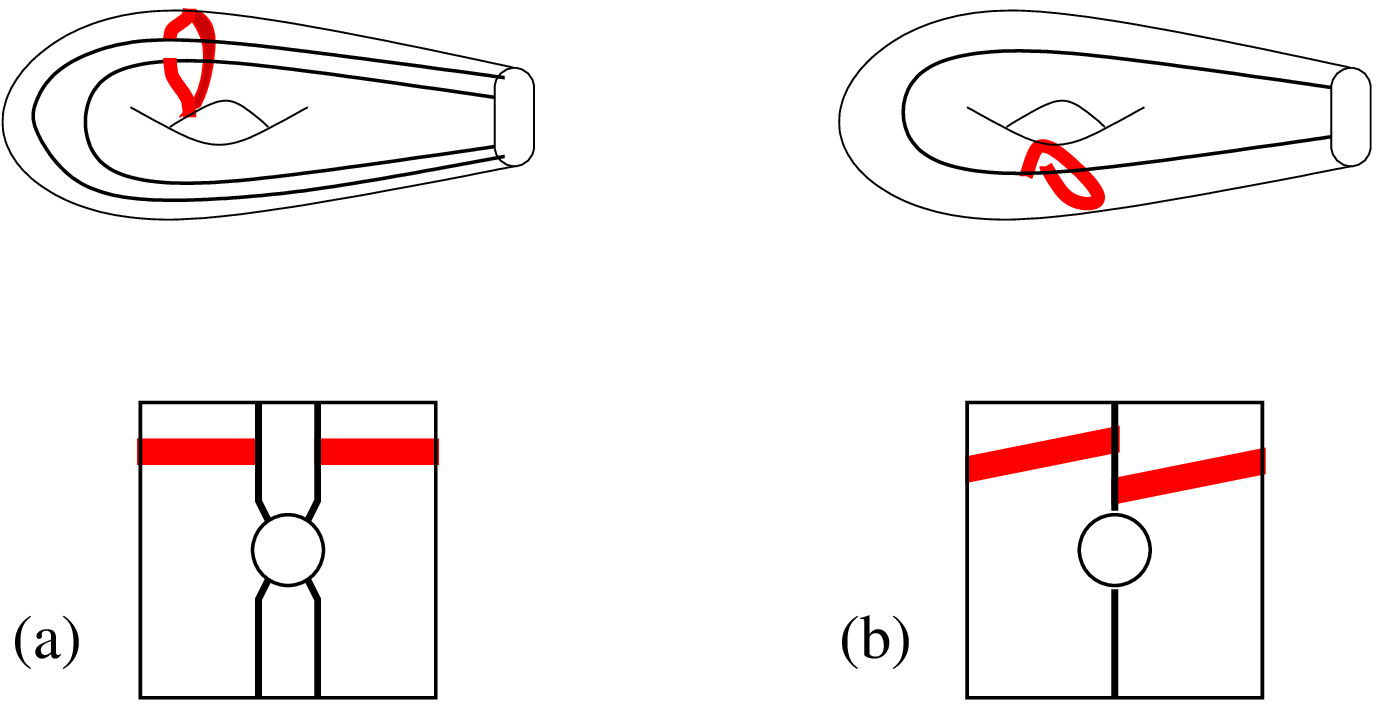,height=6.5cm}}\ \\ \
\centerline{Figure 2.25;  saddle (a) called an $e$-essential and (b) called $o$-essential saddle}\ \\

If the number of curves (on each non-critical level and of
each slope) is even then an o-essential saddle can not 
occur and we deal with the case (c)(i) of Theorem 2.13 
studied in \cite{F-H} (we allow $S$ to be nonorientable).
Let the number of the curves of some slope be odd. 
Because $\phi$ is hyperbolic, the o-essential saddle has to 
occur.  We can eliminate circles (the only possible ones 
are trivial) near o-essential saddles by isotopy of $S$. 
Let $F_{t_{1}},...,F_{t_{k}}$ be levels just below $o$-essential saddles 
and $F_{t'_{1}},...,F_{t'_{k}}$, just above the o-essential saddles.  There are no
more saddles between $F_{t_{i}}$ and $F_{t'_{i}}$ ($i=1,2,..,k$).  $F_{t_{i}} \cap S$ 
(resp $F_{t'_{i}} \cap S$) consists of one arc, say $\gamma_{i}^{-}$ (resp. $\gamma_{i}^{+}$) of slope  
$\frac{a_{i}}{b_{i}}$ (resp. $\frac{a_{i+1}}{b_{i+1}}$), which plays a role in the $i$th saddle
and $k_{i}$ ($k_{i} \geq 0$) arcs of slope $\frac{a'_{i}}{b'_{i}}$ such that $\mid det \left[\begin{array}{cc} 
a_{i} & a'_{i} \\ 
b_{i} & b'_{i} 
\end{array}\right] \mid =$ 
$ \mid det \left[\begin{array}{cc} 
a_{i+1} & a'_{i} \\ 
b_{i+1} & b'_{i} 
\end{array}\right] \mid =1$.  ($\frac{a'_{i}}{b'_{i}}$ is not uniquely determined by $\frac{a_{i}}{b_{i}}$ and 
$\frac{a_{i+1}}{b_{i+1}}$; in fact $\frac{a'_{i}}{b'_{i}}=\frac{\frac{1}{2}(a_{i+1} \pm a_{i})}{\frac{1}{2}(b_{i+1} \pm b_{i})}$). 
Consider the region between $F_{t'_{i}}$ and $F_{t_{i+1}}$ (decomposed as 
a cube as in the proof of Theorem \ref{Theorem 2.13} (a) or (b)).  After an 
appropriate choice of coordinates we can assume that $\frac{a_{i+1}}{b_{i+1}}=\frac{1}{0}$, 
$\frac{a'_{i}}{b'_{i}}=\frac{0}{1}$, and that $\frac{a'_{i+1}}{b'_{i+1}}=\frac{p}{1}$ (for some $p \geq 0$).  See Figure 2.26.\\

\centerline{\psfig{figure=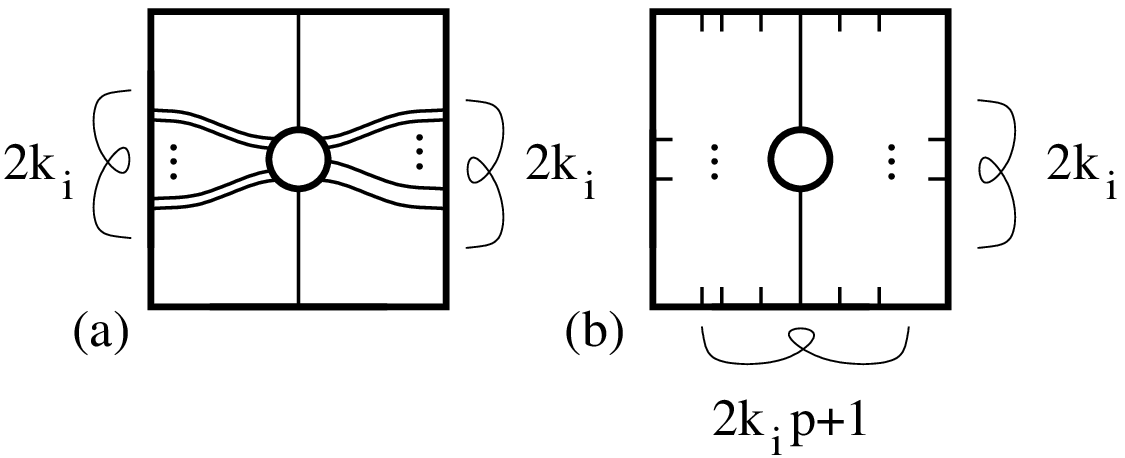,height=5.5cm}}\ \\ \
\centerline{Figure 2.26; (a) $F_{t'_{i}}$ - the bottom face of the cube; (b) $F_{t_{i+1}}$ - the top face of the cube }\

Now we use the following observation.\\
If $\gamma,\gamma_{1},...,\gamma_{2k}$ $(k \geq 0)$ are (all) curves of a given slope in 
$S \cap F_{t}$ (for some noncritical level $t$) and $\gamma$ lies in the middle 
of the curves (Figure 2.27), then if $S_{a}$ is the nearest essential saddle to level $F_{t}$ (from         
above or below) then either $\gamma$ is not changed in $S_{a}$ (it is always
the case if $k \textgreater 0$) or $S_{a}$ is o-essential.  Furthermore the
condition that $\gamma$ is in the middle is preserved. From this                                 
observation and Proposition 2.1 of \cite{F-H} it follows that, after  
isotopy, one can assume that in the region between $F_{t'_{i}}$ and
$F_{t_{i+1}}$ 

\centerline{\psfig{figure=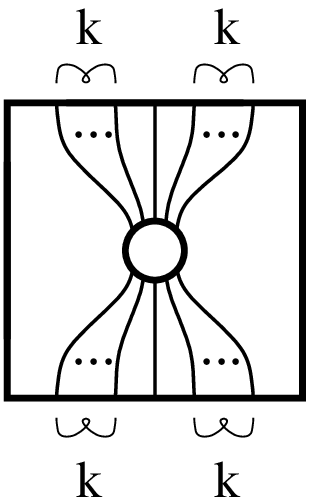,height=6.1cm}}\ \\ \
\centerline{Figure 2.27;  }\
\ \\
the curve $\gamma_{i}^{+}$ (and $\gamma_{i+1}^{-}$) is not ``involved" in any  
saddles.  Therefore each connected component of the lifting  
of S to $\tilde{S}$ in $F \times \R$ (which contains an o-saddle) is of the  
forms $\tilde{S}_{\gamma}(\varepsilon_{1},...,\varepsilon_{k})$.  Now we have two possibilities for  
connected S:
\begin{enumerate}
\item[(i)] on some (so each) non-critical level $t$, $F_{t} \cap S$ consists  
of one curve.  Then we deal with the case (ii) of Theorem \ref{Theorem 2.13}(c),
\item[(ii)] on some (so each) non-critical level t, $F_{t} \cap S$ consists  
of two curves.  Then two components of S are interchanged  
by $\phi$ and we deal with the case (iii) of Theorem \ref{Theorem 2.13}(c).  
\end{enumerate}
The proof of Theorem \ref{Theorem 2.13}(c) will be completed if we show that  
for a non-minimal $\gamma$, S is not incompressible, $\partial$-incompressible.

Consider the case of $S_{\gamma}(\varepsilon_{1},...,\varepsilon_{k})$ with $\gamma$ determined by vertices 
$...\frac{a_{-1}}{b_{-1}},\frac{a_{0}}{b_{0}},\frac{a_{1}}{b_{1}},...$ .  There are just two possibilities for 
successive o-saddles yielding $\frac{a_{i}}{b_{i}}=\frac{a_{i+2}}{b_{i+2}}$, up to a change of coordinates:\\

\centerline{\psfig{figure=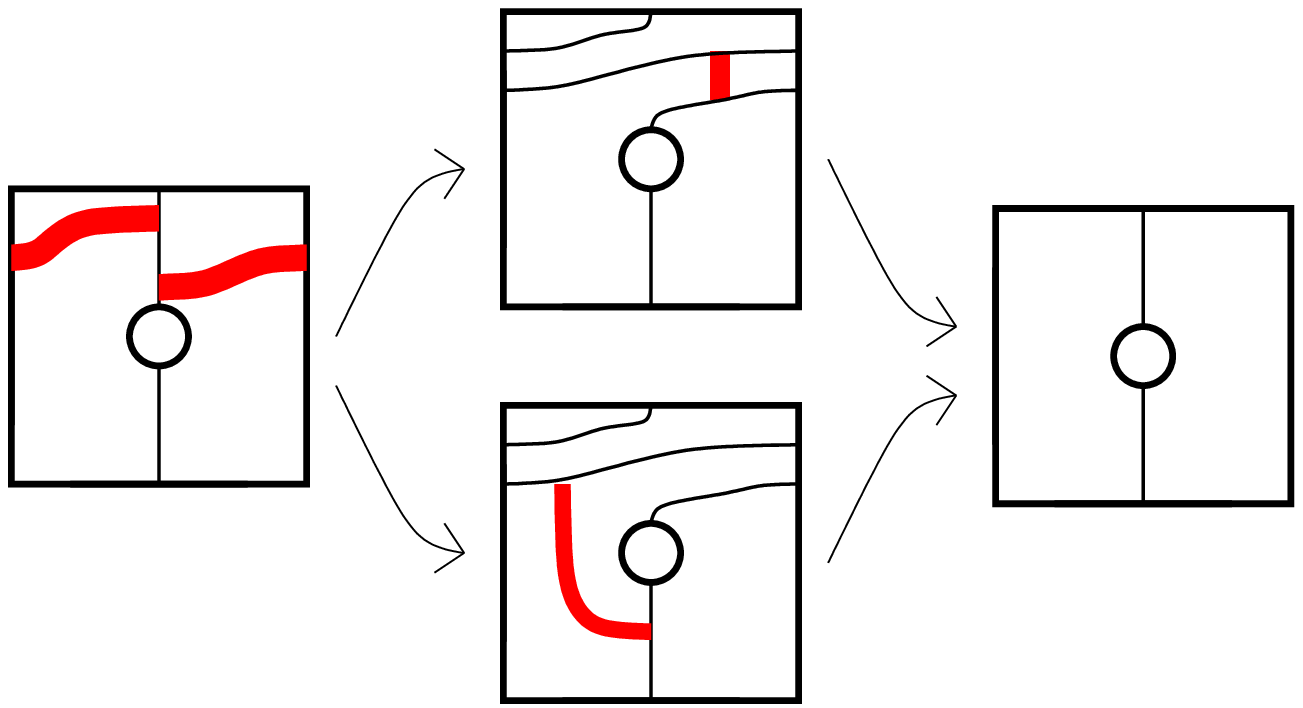,height=5.0cm}}\ \\ \
\centerline{Figure 2.28; $\gamma$ moving for and back }\

In the first of these two sequences, $S$ is clearly compressible.
In the second, $S$ is $\partial$-compressible.  This can be seen after one  
puts both saddles on the same level (Figure 2.29).\\

\centerline{\psfig{figure=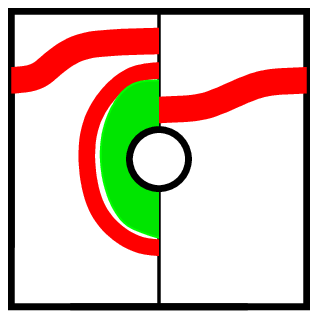,height=3.0cm}}\ \\ \
\centerline{Figure 2.29; $\partial$-compressing disk  }\

Consider the case of $S_{\gamma}^{sp}$.  If $\gamma$ is not minimal then $\gamma_{1}$ or 
$\gamma_{2}$ is not minimal.  Then we can find consecutive vertices in  
$\gamma$ such that $(\frac{a_{i}}{b_{i}},\frac{c_{i}}{d_{i}})=(\frac{a_{i+2}}{b_{i+2}},\frac{c_{i+2}}{d_{i+2}})$ (assume for simplicity that 
the first slope is going for and back so the second is unchanged $\frac{c_{i}}{d_{i}}=\frac{c_{i+1}}{d_{i+1}}=\frac{c_{i+2}}{d_{i+2}}$).   
Because $\phi$ is hyperbolic so $\gamma$ has the period at least $2$. Therefore \\

 the surface corresponding to $(\frac{a_{i}}{b_{i}},\frac{c_{i}}{d_{i}}),(\frac{a_{i+1}}{b_{i+1}},\frac{c_{i+1}}{d_{i+1}}), 
(\frac{a_{i+2}}{b_{i+2}},\frac{c_{i+2}}{d_{i+2}})$ allows the obvious 
compressing disk which is a compressing disk of $S_{\gamma}^{sp}$ (compare 
the upper part of Figure 2.28).\\
This ends the proof of Theorem \ref{Theorem 2.13}; it still remains to prove Proposition \ref{Proposition 2.16}.
\begin{remark}\label{Remark 2.20}
If we allow S to be disconnected then we can get 
more incompressible, $\partial$-incompressible surfaces, for example 
surfaces $S_{\gamma}^{sp}$ in $M_{\phi}$ with $\phi(\gamma)=\gamma$ (compare Example \ref{Example 2.22}). 
\end{remark}
Proof of Proposition \ref{Proposition 2.16}:
\begin{enumerate}
\item[(a)] A case of a closed surface.\\
Consider the following properties of a closed, connected 
surface $S_{\gamma}^{c}$ in $M_{\phi}$:
\begin{enumerate}
\item[(i)] Each circle of the intersection of $S_{\gamma}^{c}$ with a non-critical fiber $F_{t}$ which is trivial 
in $F_{t}$ bounds a disk in $S_{\gamma}^{c}$. 
\item[(ii)] $S_{\gamma}^{c} \cap F_{t}$, where $F_{t}$ is a non-critical fiber does not contain circles parallel to $\partial F_{t}$. 
\item[(iii)] On each non-critical level $F_{t}$ there is exactly one 
slope (and it is represented by an odd number of circles) 
and the sequence of these slopes in $S_{\gamma}^{c}$ traces out the
vertex sequence of the given minimal, invariant edge-path $\gamma \subset \W$. 
\end{enumerate}


We claim that properties (i)-(iii) are preserved by any isotopy of $S_{\gamma}^{c}$.  
If this is so then, the proposition follows (compare \cite{F-H}).  
For suppose $S_{\gamma}^{c}$ was compressible.  Let $D$ be a compressing disk  
$D \cap S_{\gamma}^{c}= \partial D$.  A small sub-disk $D'$ can be isotoped to lie in a fiber  
$F_{t}$.  The shrinking of $D$ to $D'$ extends to an isotopy of $S_{\gamma}^{c}$ to  
$S_{\gamma}^{'c}$.  Condition (i) implies that $\partial D'$ bounds a disk in $S_{\gamma}^{'c}$, so 
$\partial$D bounds a disk in $S_{\gamma}^{c}$.\\
To prove the claim, consider a generic isotopy of $S_{\gamma}^{c}$.  At any  
time during this isotopy, the projection to $S^{1}$ will be a Morse 
function, except for the following isolated phenomena:
\begin{enumerate}
\item[(A)] a saddle and a local maximum (or minimum) are introduced 
or canceled in a region containing no other critical points,
\item[(B)] a pair of critical points interchange levels.
\begin{enumerate}
\item[(A)] cannot affect conditions (i)-(iii), 
\item[(B)] cannot affect conditions (i)-(iii) when one or both 
of the nondegenerate critical points are of index 0 or 2. 
\end{enumerate}
\end{enumerate}
Thus the only case left to check is when two saddles interchange heights.  
Each of the two saddles has one of the forms:\\

\centerline{\psfig{figure=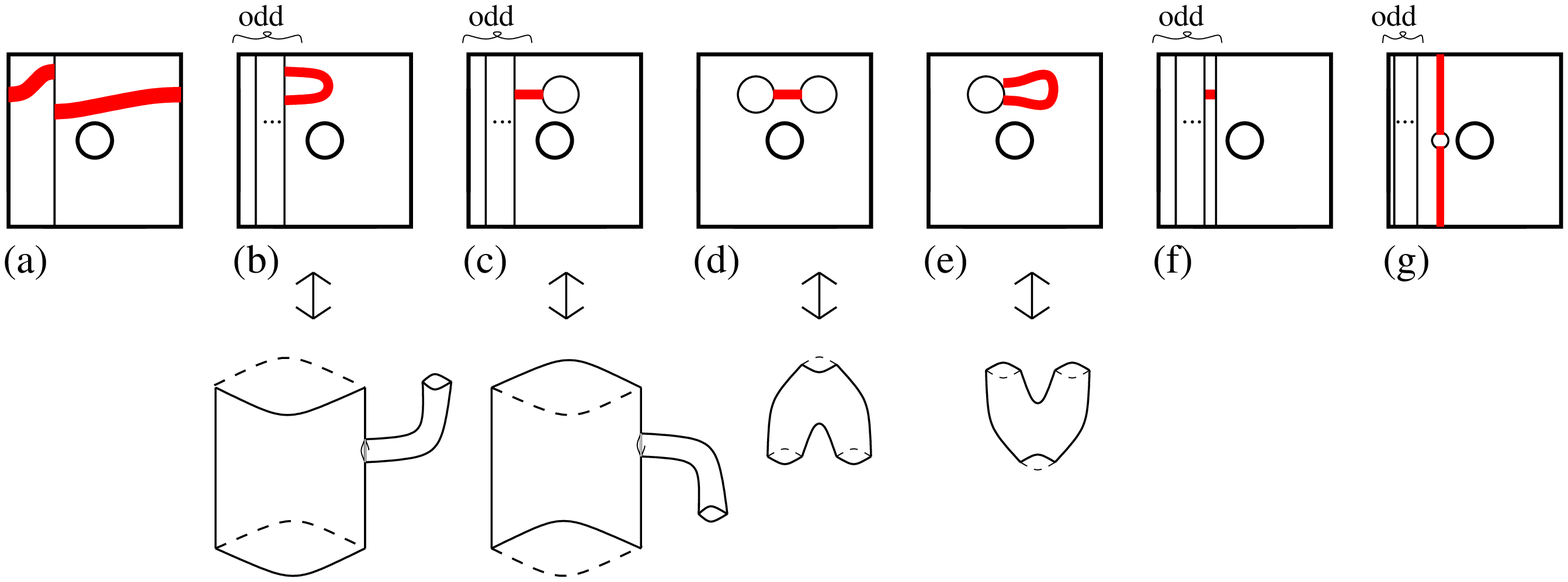,height=6.0cm}}\ \\ \
\centerline{Figure 2.30;  saddles which can interchange heights}\


If both saddles are of type (a), up to the level preserving isotopy and change of coordinates, then we have the following two 
possibilities:\\

\centerline{\psfig{figure=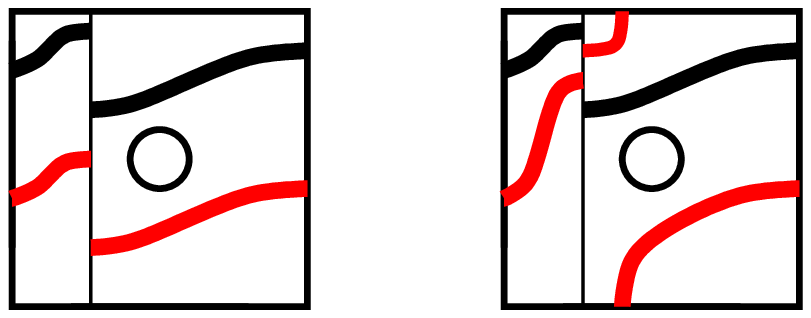,height=3.0cm}}\ \\ \
\centerline{Figure 2.31;  two cases when saddles of type (a) commute}\

Interchanging heights of the saddles in these cases preserves (i)-(iii).  
In all other cases one can check that the conditions (i)-(ii) are preserved 
(to prove that condition (ii) is preserved it is useful to consider the fundamental 
group of a part of $S_{\gamma}^{c}$ consisting of a segment containing both saddles).\\
For (iii), if neither saddle before the interchange was essential then the same is true after the interchange, 
so the edge-path is preserved (however a pair of circles with defined slope 
can be created or canceled; see Figure 2.32).\\

\centerline{\psfig{figure=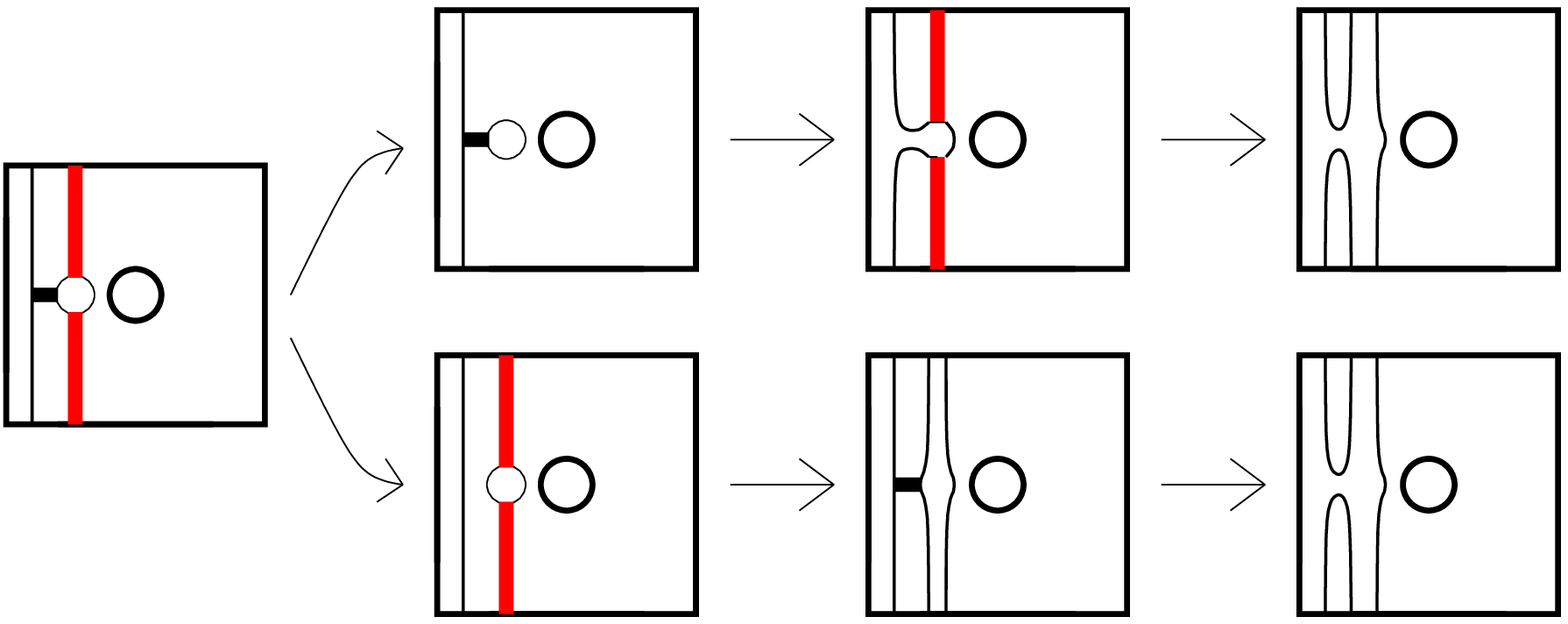,height=5.0cm}}\ \\ \
\centerline{Figure 2.32;  }\

If just one saddle before the interchange is essential it holds
after the interchange also (it follows from the fact that $\gamma$
lies on the tree; see Fact 2.2).  Thus the edge-path is 
preserved and hence the condition (iii).  It ends the part (a)  
of the proof of Proposition 2.16.
\item[(b)] The case of one horizontal boundary component. 
Consider the following properties of the connected surface $S$ 
with exactly one, horizontal boundary component:
\begin{enumerate}
\item[(i)] Each circle of the intersection of $S$ with a non-critical 
fiber $F_{t}$ which is trivial in $F_{t}$ bounds a disk in $S$,
\item[(ii)] Each circle of the intersection of $S$ with a non-critical 
fiber $F_{t}$ which is parallel to $\partial F_{t}$ in $F_{t}$ is parallel to $\partial S$ 
in $S$,
\item[(iii)] On each non-critical level $F_{t}$ there is exactly one 
slope and it is represented by an odd number of circles. 
The sequence of these slopes (in $S$) traces out the vertex sequence 
of the given minimal edge-path $\gamma \subset \W$.
\end{enumerate}
We can prove, similarly as in the previous case, that 
properties (i)-(iii) are preserved by any isotopy of $S$ (rel $\partial$S).  Thus the
part (b) of the proof of Proposition 2.16 is completed. \ \ \ \ \ \ $\square$
\end{enumerate}
\begin{remark}\label{Remark 2.21}
(a) and (b) can be derived from the fact that the 
incompressible surfaces in $T^{2}$ bundle over $S^{1}$ with a hyperbolic monodromy map are 
classified up to isotopy by invariant, minimal edge-paths in $\W$ (and element of $\Z_{2}$); see 
Theorem \ref{Theorem 2.3} and Remark \ref{Remark 2.4}.
\end{remark}


\begin{enumerate}
\item[(c)] The case of vertical boundary components.\\
Consider the following properties of a connected surface, $S$, with vertical boundary components:
\begin{enumerate}
\item[(i)] $S$ has no trivial arcs on any non-critical level,
\item[(ii)] Each non-critical level circle of $S$ bounds a disk in $S$,
\item[(iii)] On each non critical level there is exactly one arc of
a defined slope, and the slope sequence of $S$ traces out 
the vertex sequence of a given minimal edge-path $\gamma \subset \W$.
\end{enumerate}
We claim that properties (i)-(iii) are preserved by any isotopy of $S$
(rel $\partial S$) and that the only operation (under generic isotopy)
which could reverse the type of an essential saddle is interchanging the relative heights of this saddle and another 
saddle and we are in the same situation as Example \ref{Example 2.15}.  
The considerations are similar to those of part (a) and (b) and \cite{H-T,F-H}, and we omit them.\\
From the considerations above it follows that $S_{\gamma}(\varepsilon_{1},...,\varepsilon_{k})$ 
is incompressible and now we will prove $\partial$-incompressibility. 
If $b(S_{\gamma}(\varepsilon_{1},...,\varepsilon_{k}))=2$ then $\partial$-incompressibility follows from 
condition (iii) and Proposition \ref{Proposition 2.8}.  \\
If $b(S_{\gamma}(\varepsilon_{1},...,\varepsilon_{k}))=1$,
we consider the lifting of $S_{\gamma}(\varepsilon_{1},...,\varepsilon_{k})$ to $M_{\phi^{2}}$ and the lifted 
manifold has two boundary components and it is incompressible, so $\partial$-incompressible, 
so $S_{\gamma}(\varepsilon_{1},...,\varepsilon_{k})$ is $\partial$-incompressible.\\
To prove that $S_{\gamma}^{sp}$ is incompressible, $\partial$-incompressible we  
consider the lifting of this surface to $M_{\phi^{2}}$.  Some 
connected component of the lifted surface is of type 
$S_{\gamma}(\varepsilon_{1},...,\varepsilon_{k})$ where $\gamma_{1}$ is the $\phi^{2}$-invariant, minimal edge-path 
in $\W$ associated to the special edge-path $\gamma$ (see Definition 2.11).  Because $S_{\gamma}(\varepsilon_{1},...,\varepsilon_{k})$ is incompressible, 
$\partial$-incompressible so $S_{\gamma}^{sp}$ is incompressible, $\partial$-incompressible.
This ends the proof of Proposition \ref{Proposition 2.16}. \ \ \ \ \ \ \ \ \ \ $\square$
\end{enumerate}


\begin{proof} (Proof of Proposition \ref{Proposition 2.17}) \\
  The calculations are standard.  The only one which is more troublesome is the computation of 
$sl(S_{\gamma}(\varepsilon_{1},...,\varepsilon_{k}))$ and $sl(S_{\gamma}^{sp})$.  The computation of 
$sl(S_{\gamma}(\varepsilon_{1},...,\varepsilon_{k}))$ reduces to the computation of $sl(S_{\gamma'})$ 
where $S_{\gamma'}$ is the boundary of a tubular neighborhood of 
$S_{\gamma}(\varepsilon_{1},...,\varepsilon_{k})$; see Definition \ref{Definition 2.14}.  The computation of 
$sl(S_{\gamma}^{sp})$ reduces to the computation of $sl(S_{\gamma'})$ where $S_{\gamma'}$ 
is the boundary of a tubular neighborhood of some  
connected component of the lifting of $S_{\gamma}^{sp}$ to $M_{\phi^{2}}$. 
\end{proof}

One could hope to extend our classification to the case of  
incompressible, but not $\partial$-incompressible surfaces, however some  
additional difficulties are involved (compare Example \ref{Example 2.22})
and we stop on Theorem \ref{Theorem 2.13} and Proposition \ref{Proposition 2.8}.\\
  If we drop the assumption about hyperbolicity of the monodromy map, we  
will deal either with a periodic monodromy so with a Seifert  
fibered space (for this case see \cite{P-2}) or with a reducible  
monodromy map \cite{T}.  The latter case may be studied by using  
Proposition \ref{Proposition 2.8} and the knowledge about incompressible  
surfaces in a 2-punctured disk bundle over $S^{1}$ (compare \cite{CJR} 
and \cite{P-2}).\\
In \cite{P-1} we study with details the case of nonorientable, 
incompressible surfaces of genus 3 embedded in manifolds 
obtained from punctured-torus bundles over $S^{1}$ by capping  off the torus in the boundary.
 \ \\ \ \\

\section{Example of surfaces in $M_{\phi^k}$}\label{Section 3}
We illustrate our classification theorems on the rather general example.
\begin{example}\label{Example 2.22}
Consider manifold $M_{\phi^{k}}$ where\\
$\phi=\left[\begin{array}{cc} 
5 & 2 \\ 
2 & 1 
\end{array}\right] = \left[\begin{array}{cc} 
1 & 1 \\ 
0 & 1 
\end{array}\right]^{2} \centerdot \left[\begin{array}{cc} 
1 & 0 \\ 
1 & 1 
\end{array}\right]^{2} = \bar{\alpha}^{2}\beta^{2}$; see Figure 3.1  
(here $\bar\alpha = 
\left[\begin{array}{cc} 
1 & 1 \\ 
0 & 1 
\end{array}\right]$ and $\beta= \left[\begin{array}{cc} 
1 & 0 \\ 
1 & 1 
\end{array}\right]$ are standard generators of $SL(2,Z)$; compare \cite{F-H}).\\

\centerline{\psfig{figure=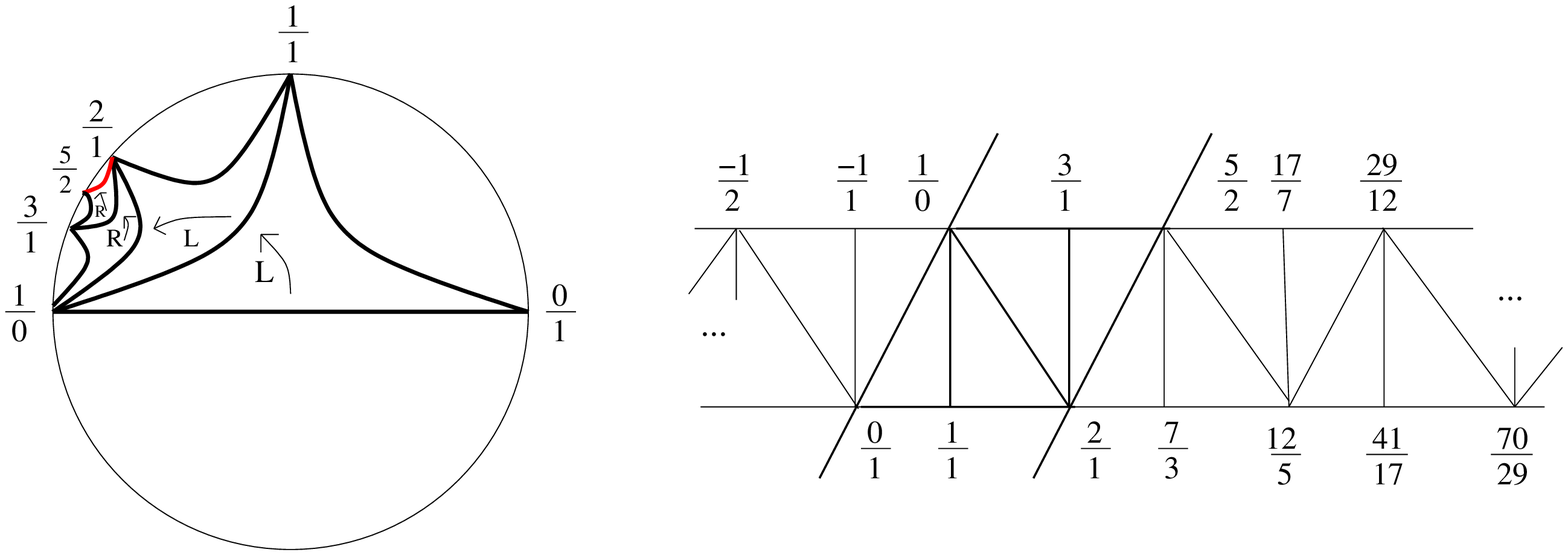,height=6.0cm}}\ \\ \
\centerline{Figure 3.1;  part of $PSL(2\Z)$ diagram}\

Because $\phi^{k}$, modulo $\Z_{2}$, has three eigenvalues we have three 
minimal, invariant edge-paths in $\W$ ($\gamma_{i} \subset \W_{i}$, i=1, 2 or 0) 
(compare \cite{F-H}; the first version).  Therefore we have three (up to 
isotopy) closed, non-parallel to the boundary, incompressible 
surfaces in $M_{\phi^{k}}$ (see Theorem 2.13a) and Proposition \ref{Proposition 2.16}):\\
$S_{\gamma_{1}}^{c}$, $S_{\gamma_{2}}^{c}$, and $S_{\gamma_{0}}^{c}$. 
Now consider each $\gamma_{i}$ independently. 
\begin{enumerate}
\item[(1)] Consider $\gamma_{1} \subset \W_{1}$.  It is determined by the vertices:\\
$...,\frac{a_{0}}{b_{0}}=\frac{0}{1},\frac{a_{2}}{b_{2}}=\frac{2}{1},\frac{a_{4}}{b_{4}}=\frac{12}{5},...,\frac{a_{2k}}{b_{2k}}= 
\left[\begin{array}{cc} 
5 & 2 \\ 
2 & 1 
\end{array}\right]^{k} (\frac{0}{1}),...$\\
we have $\frac{a_{2i+2}}{b_{2i+2}}=\frac{a_{2i-2}+2\kappa a_{2i}}{b_{2i-2}+2\kappa b_{2i}}$ where $\kappa =-3$ so ($2\kappa \neq \pm 2$), so\\
each symbol $[\varepsilon_{1},...,\varepsilon_{k}] \in (\Z_{2})^{k}$ gives us a different surface 
$S_{\gamma}(\varepsilon_{1},..,\varepsilon_{k})$ (see Definition 2.14 and Proposition 2.16).\\
Slopes are: $sl(S_{\gamma_{1}}(\varepsilon_{1},...,\varepsilon_{k}))=\frac{1}{2} \Sigma \varepsilon_{i}$.  Boundary of a tubular neighborhood 
of each $S_{\gamma_{1}}(\varepsilon_{1},...,\varepsilon_{k})$ is incompressible (compare Definition \ref{Definition 2.14}).\\
Now let us assume that $\Sigma \varepsilon_{i}$ is odd, so $b(S_{\gamma_{1}}(\varepsilon_{1},...,\varepsilon_{k}))=1$\\
  Consider the construction from Proposition \ref{Proposition 2.8} (as in 
Observation \ref{Observation 2.19}) with $S_{0}$ a Mobius band with a hole;
$S_{0} \subset [0,1] \times T^{2} ([0,1] \times T^{2}$ will be glued to $M_{\phi^{k}}$ along $\{0\} \times T^{2}$ and  
		$\partial M_{\phi^{k}})$.  $S_{0}$ is determined uniquely by the slopes $\frac{1}{2} \Sigma \varepsilon_{1}$ on  
$\{0\} \times T^{2}$ and $\frac{1}{0}$ on $\{1\} \times T^{2}$.  It leads us to the surface $\bar{S}_{\gamma_{1}}$ (see 
Observation 2.19) independently of a choice of $(\varepsilon_{1},...,\varepsilon_{k})$ 
with $\Sigma \varepsilon_{1}$ odd.  Hence we obtain the following examples:
\begin{itemize}
\item[(i)] if 
	$(\varepsilon_{1},...,\varepsilon_{k}) \neq (\varepsilon'_{1},...,\varepsilon'_{k})$ and $\Sigma \varepsilon_{i} 
		= \Sigma \varepsilon'_{i}= \mbox{ odd number}$\\
then the construction gives us examples of incompressible, 
$\partial$-incompressible surfaces in $M_{\phi^{k}}$ which are not isotopic 
but which after adding the "collar'' ($I \times T^{2}, S_{0}$) become isotopic 
(however still incompressible).
\item[(ii)] if $\Sigma \varepsilon_{i} \neq \Sigma \varepsilon'_{i}$ (both numbers odd) and we add to the surface 
$S_{\gamma_{1}}(\varepsilon_{1},...,\varepsilon_{k})$ the "collar'' ($[0,2] \times T^{2}, S_{1}$) where $S_{1}$ is the unique 
incompressible Klein bottle in $[0,2] \times T^{2}$ with two holes given 
by slopes $\frac{1}{2}(\Sigma \varepsilon_{i})$ in $\{0\} \times T^{2}$ and $\frac{1}{2} \Sigma \varepsilon'_{i}$ in $\{2\} \times T^{2}$ (see 
Theorem \ref{Theorem 2.3}).  In such a way we construct the compressible surface promised in Proposition \ref{Proposition 2.8}.  
To see that the constructed surface (say $S$) is compressible we can use equality:
$$S=S_{\gamma_{1}}(\varepsilon_{1},...,\varepsilon_{k}) \cup S_{1}=S_{\gamma_{1}}(\varepsilon_{1},...,\varepsilon_{k}) \cup S'_{1} \cup S_{1}'',$$ 
		where 
$S'_{1}$ in $[0,1] \times T^{2}$ is given by the slopes $\frac{1}{2} \Sigma \varepsilon_{i}$ in $\{0\} \times T^{2}$ and $\frac{1}{0}$
in $\{1\} \times T^{2}$ and $S_{1}''$ in $[1,2] \times T^{2}$ is determined by the slopes $\frac{1}{0}$ 
in $\{1\} \times T^{2}$ and $\frac{1}{2} \Sigma \varepsilon'_{i}$ in $\{2\} \times T^{2}$.  S is isotopic to  
$S_{\gamma'_{1}}(\varepsilon_{1},...,\varepsilon'_{k}) \cup S_{2} \cup S_{1}''$ where $S_{2}$ in $[0,1] \times T^{2}$ is determined 
by the slopes $\frac{1}{2} \Sigma \varepsilon'_{1}$ in $\{0\} \times T^{2}$ and $\frac{1}{0}$ in $\{1\} \times T^{2}$.  S is compressible 
because $S_{2} \cup S_{1}''$ is compressible (nonorientable surface 
in $[0,2] \times T^{2}$ with the same slopes in $\{0\} \times T^{2}$ and $\{2\} \times T^{2}$; see Theorem \ref{Theorem 2.3}).
\end{itemize}
\item[(b)] $\gamma_{2} \subset \W_{2}$, $\gamma_{2}$ is determined by the vertices\\
$\ldots,\frac{a_{0}}{b_{0}}=\frac{1}{0},\frac{a_{2}}{b_{2}}=\frac{5}{2},\frac{a_{4}}{b_{4}}=
\frac{29}{12},\ldots,\frac{a_{2k}}{b_{2k}}= \phi^{k}(\frac{1}{0}),\ldots$ (see Figure 3.2).\\

\centerline{\psfig{figure=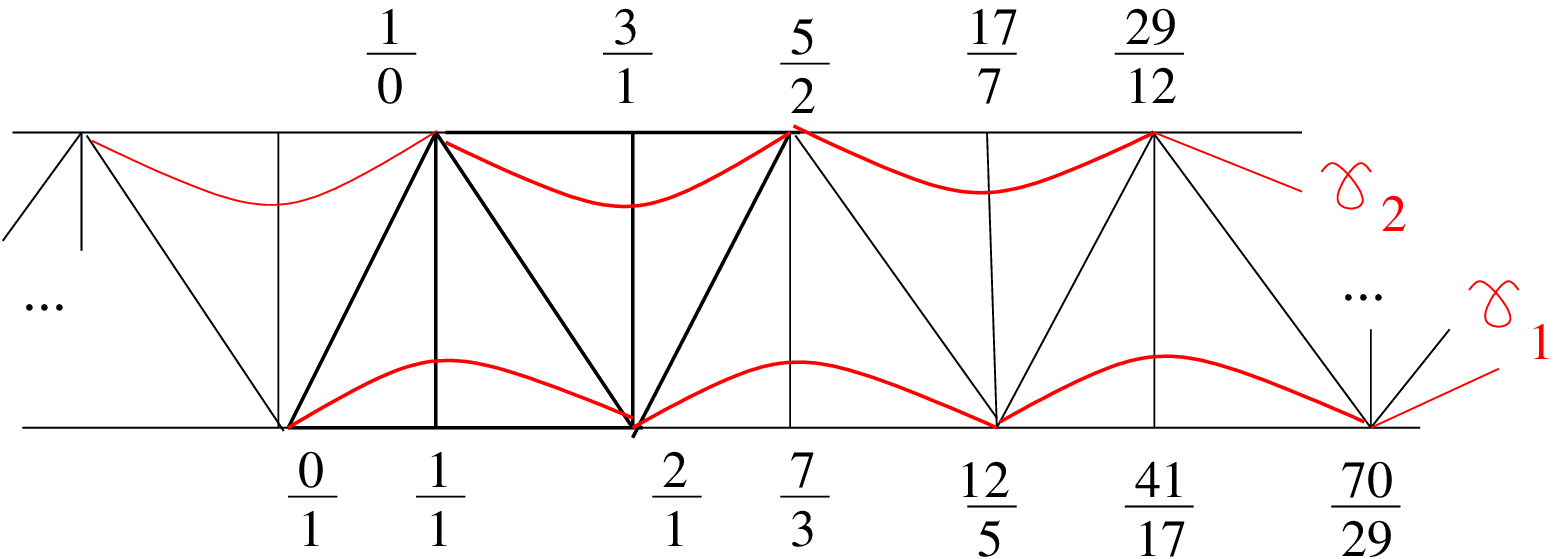,height=4.5cm}}\ \\ \
\centerline{Figure 3.2;  }\

There are only trivial relations among symbols $(\varepsilon_{1},...,\varepsilon_{k})$, 
similar to the case of $\gamma_{1}$, and $sl(S_{\gamma_{2}}(\varepsilon_{1},...,\varepsilon_{k})=-\frac{1}{2} \Sigma \varepsilon_{i}$. 
Incompressible surfaces $S_{\gamma_{1}}(0,...,0)$ and $S_{\gamma_{2}}(0,...,0)$
are disjoint and non-isotopic but the boundaries of their tubular neighborhoods are parallel.


\item[(c)] $\gamma_{0} \subset \W_{0}$; see Figure 3.3.\\

\centerline{\psfig{figure=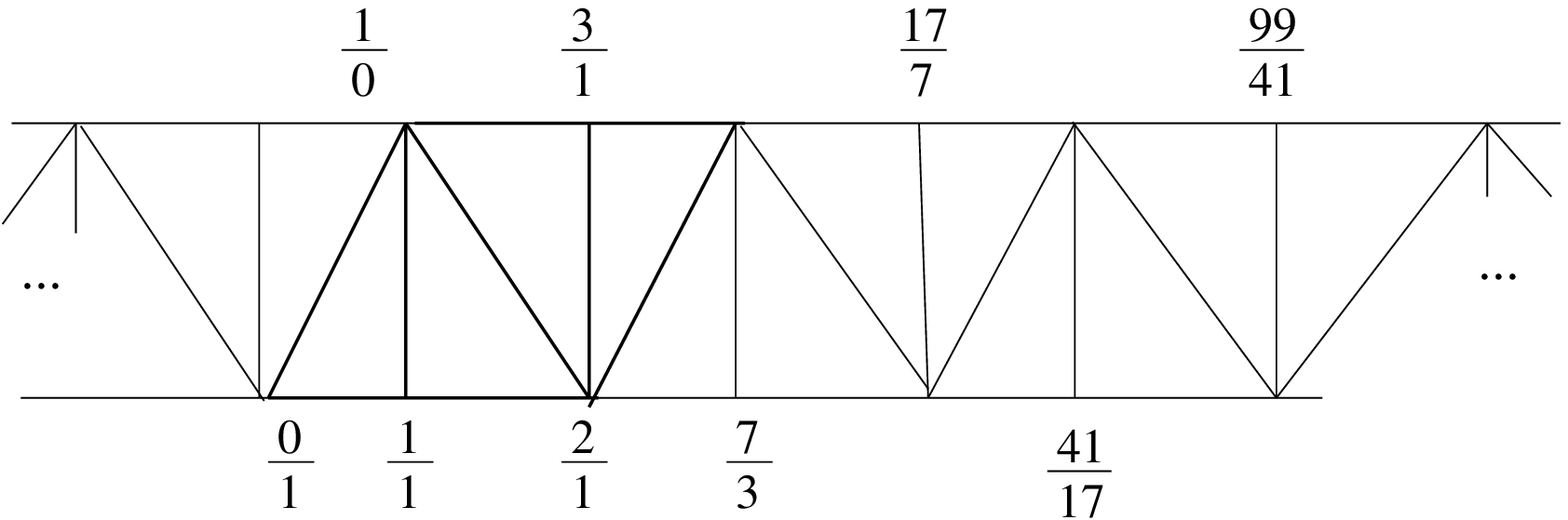,height=4.5cm}}\ \\ \
\centerline{Figure 3.3;  }\

$\gamma_{0}$ is determined by the vertices\\
$...\frac{a_{0}}{b_{0}}=\frac{1}{1},\frac{a_{2}}{b_{2}}=\frac{3}{1},\frac{a_{4}}{b_{4}}=\frac{7}{3}, \frac{a_{6}}{b_{6}}=\frac{17}{7},...$ 
and $\frac{a_{2i+2}}{b_{2i+2}}=\frac{a_{2i-2}+2a_{2i}}{b_{2i-2}+2b_{2i}}$ so 
Definition 2.14 gives us relations among symbols $(\varepsilon_{1},\varepsilon_{2},...,\varepsilon_{2k})$ 
which identify symbol $(...0^{i-th},0...)$ with $(...1^{i-th},1...)$. 
It gives us $2k+1$ non-isotopic surfaces $S_{\gamma_{0}}(\varepsilon_{1},\varepsilon_{2},...,\varepsilon_{2k})$.\\
$sl(S_{\gamma_{0}}(\varepsilon_{1},\varepsilon_{2},...,\varepsilon_{2k}))=\frac{1}{2} \Sigma (-1)^{i+1} \varepsilon_{i}$.  
$S_{\gamma_{0}}(\varepsilon_{1},\varepsilon_{2},...,\varepsilon_{2k})$\\
is not $\pi_{1}$-injective (the boundary of a tubular neighborhood of 
$S_{\gamma_{0}}(\varepsilon_{1},\varepsilon_{2},...,\varepsilon_{2k})$ is compressible; see Figure 3.3).
\item[(d)] Each minimal, $\phi^{k}$-invariant edge-path in the diagram of 
$PSL(2,\Z)$ is of even period, so it leads to an orientable manifold (see Proposition \ref{Proposition 2.17}).
\end{enumerate}
\end{example}

\section{Incompressible surfaces and skein modules}\label{Section 4}
Last time, before Maite-fest, I visited Zaragoza in February 1986; I was already then thinking about the Jones polynomial 
and its generalizations (e.g. Conway algebras). Soon after, in April 1987, I discovered skein modules of 3-manifolds \cite{P-s1}.
Immediately I thought that incompressible surfaces have an important role in creating torsion in skein modules \cite{H-P,P-s2,P-s3}.
In particular I asked:
\begin{conjecture}\label{IX.4.6}\ \\ If $M$ is a submanifold of a
rational homology sphere and it does not contain a closed, oriented
incompressible surface then its Homflypt skein module ${\cal S}_3(M)$
is free and isomorphic to the symmetric tensor algebra over module
spanned by conjugacy classes of nontrivial elements of the fundamental
group, ${\cal S}_3(M) = {\bf S}R\hat\pi^o$.
\end{conjecture}

I will leave it to readers to think of this and other possible relations of incompressible surfaces and torsion of skein modules. 

\section{Acknowledgements}\label{Section 5}
I would like to express my gratitude to Prof. Joan Birman 
whose assistance enabled me to  come to study at Columbia and to complete my PhD thesis of which this paper is a part.\\
I would like to thank Maite Lozano and Jos\'e Montesinos for giving me the opportunity to talk about my thesis when I visited Zaragoza 
in 1982\footnote{I defended my thesis August 31, 1981 and in September left for Poland. I submitted the paper 
before leaving. The martial law called officially ``State of War" came to Poland in  December 1981. When I got, with some delay, positive referee report, 
but being asked, justly so, for improvements in presentation, I was already distracted from thinking of my thesis. 
Jos\'e Montesinos kindly invited me to visit him in Zaragoza and in Fall 1982 I got my passport. In October 1982 I gave series of 
talks and that is how I met first time Maite. With the help of an inquisitive audience I noticed that I ``lost" a family of incompressible surfaces 
in my dissertation. Already in Spain I repaired my work but with publication it waited till now. I think it is the proper paper to 
celebrate Maite's seventieth birthday. Specially so because after my visit to Zaragoza we started collaboration and we 
analyzed incompressible surfaces in the complement of a closed 3-braid. The first part of this research was published \cite{LP-1},
while the second still awaits to be typed \cite{LP-2}.}.\\
I was partially supported by the Simons Collaboration Grant-316446 and CCAS Dean's Research Chair award.

\ \\
Department of Mathematics,\\
The George Washington University,\\
Washington, DC 20052\\
e-mail: {\tt przytyck@gwu.edu},\\
and University of Gda\'nsk, Poland

\begin{thebibliography}{99}

\bibitem
{B-W}
G.~E.~Bredon, J.~W.~Wood, Non-orientable surfaces in orientable 3-manifolds,
{\it Inventiones Math.} 7, (1969), 83-100.

\bibitem
{CJR}
M.~Culler, W.~Jaco, J.~H.~Rubinstein, Incompressible surfaces in
once-punctured-torus bundles, {\it Proc. London Math. Soc.},
(3) 45, 1982, 385-419.

\bibitem
{F-H}
W.~Floyd, A.~Hatcher, Incompressible surfaces in punctured-torus
bundles, {\em Topology and its applications}, 13, (1982), 263-282.

\bibitem
{H-T}
A.~Hatcher, W.Thurston, Incompressible surfaces in 2-bridge knot
complements, {\em Inventiones Math.}, 79 (1985), 225-246.

\bibitem{H-P}
J.~Hoste, J.~H.~Przytycki, 
A survey of skein modules of 3-manifolds,
in  Knots 90, Proceedings of the International Conference on Knot
Theory and Related Topics, Osaka (Japan), August 15-19, 1990, Editor
A.~Kawauchi, Walter de Gruyter 1992, 363-379.

\bibitem{J-P}
	W.~Jakobsche, J.~H.~Przytycki, {\it Topology of 3-dimensional manifolds},
		Warsaw University Press, 1987, in Polish.

\bibitem{LP-1} M.~Lozano, J.~H.~Przytycki, 
Incompressible surfaces in the exterior of a closed 3 braid.
I. Surfaces with horizontal boundary components,
{\it Math. Proc. Cambridge Phil. Soc.}, 98,
1985, 275-299.

\bibitem{LP-2} M.~Lozano, J.~H.~Przytycki,
Incompressible surfaces in the exterior of a closed 3 braid.
II. Surfaces with vertical boundary components, in preparation.

\bibitem{MKS} W.~Magnus, A.~Karrass, D.~Solitar, Combinatorial group theory: Presentations of groups in 
	terms of generators and relations, 1966; Second revised edition, Dover Publications, INC, New York 1976.

\bibitem {P-0}
J.~H.~Przytycki, Incompressible surfaces in $3$-manifolds, Ph.D. dissertation, Columbia University, 1981;  thesis advisor: Professor Joan Birman. 

\bibitem
{P-1}
J.~H.~Przytycki, Nonorientable, incompressible surfaces  of genus $3$ in
$M_{\phi(\frac{\lambda}{\mu})}$ manifolds, {\em Collectanea
Math.},  XXXIV (1), 1983 ,37-79.

\bibitem
{P-2}
J.~H.~Przytycki, Nonorientable, incompressible surfaces in Seifert fibered
spaces, unfinished manuscript, 1981.

\bibitem
{P-s1} J.~H.~Przytycki, 
Skein modules of 3-manifolds,
{\em Bull. Ac. Pol.: Math.}; 39(1-2), 1991, 91-100; \
e-print: \ {\tt arXiv:math/0611797 [math.GT]}

\bibitem
{P-s2} J.~H.~Przytycki, 
Algebraic topology based on knots: an introduction,
{\it Knots 96}, Proceedings of the Fifth International Research Institute
of MSJ, edited by Shin'ichi Suzuki, World Scientific Publishing Co., 1997,
279-297.

\bibitem
{P-s3} J.~H.~Przytycki,
Fundamentals of Kauffman bracket skein modules, {\it Kobe Math. J.},
16(1), 1999, 45-66.\
e-print: \ {\tt arXiv:math/9809113 [math.GT]}


\bibitem
{Ru}
J.~H.~Rubinstein, One sided Heegaard splitting of 3-manifolds, {\em Pacific
J.Math.} 76(1) 1978, 185-200.


\bibitem
{T}
W.~Thurston, On the geometry and dynamics of diffeomorphisms of surfaces, I,
preprint 1976; published in: {\em Bull. Amer. Math. Soc.}, 19(2), 1988,
417-431.

\bibitem
{W}
F~.Waldhausen, Eine Klasse von 3-dimensionalen Mannigfaltigkeiten I,
{\em Inventiones Math.}, 3 (1967), 308-333; {\em Inventiones Math.},
4 (1967), 87-117.

\end{thebibliography}
\end{document}